\documentclass{amsart}
\usepackage{amsmath}
\usepackage{graphics}
\usepackage{graphicx}
\usepackage{amsthm}
\usepackage{graphicx}
\usepackage{color}
\usepackage{amsfonts}
\usepackage{amssymb}
\usepackage{mathrsfs}
\usepackage{amscd}
\usepackage{mathtools}
\usepackage{multirow}
\usepackage{longtable}
\usepackage{url}
\usepackage[justification=centering]{caption}

\newcommand{\re}{\mathbb{R}}

\newcommand{\N}{\mathbb{N}}

\newcommand{\ttS}{\mathtt{S}}

\newcommand{\lmd}{\lambda}

\newcommand{\nn}{\nonumber}
\newcommand{\eps}{\epsilon}

\newcommand{\dt}{\delta}
\newcommand{\Dt}{\Delta}
\def\af{\alpha}

\def\rank{\mbox{rank}}

\newcommand{\Sig}{\Sigma}

\newcommand{\st}{\mbox{s.t.}}
\newcommand{\mt}[1]{\mathtt{#1}}

\newcommand{\reff}[1]{(\ref{#1})}

\newcommand{\mc}[1]{\mathcal{#1}}

\newcommand{\supp}[1]{\mbox{supp}(#1)}

\newcommand{\hm}{\mathit{hom}}

\newcommand{\ideal}[1]{\mathit{Ideal}[#1]}
\newcommand{\qmod}[1]{\mathit{QM}[#1]}

\newcommand{\HM}[1]{\mathcal{H}[#1]}

\newcommand{\bdes}{\begin{description}}
\newcommand{\edes}{\end{description}}

\newcommand{\bal}{\begin{align}}
\newcommand{\eal}{\end{align}}

\newcommand{\bnum}{\begin{enumerate}}
\newcommand{\enum}{\end{enumerate}}

\newcommand{\bit}{\begin{itemize}}
\newcommand{\eit}{\end{itemize}}

\newcommand{\bea}{\begin{eqnarray}}
\newcommand{\eea}{\end{eqnarray}}
\newcommand{\be}{\begin{equation}}
\newcommand{\ee}{\end{equation}}

\newcommand{\baray}{\begin{array}}
\newcommand{\earay}{\end{array}}

\newcommand{\bsry}{\begin{subarray}}
\newcommand{\esry}{\end{subarray}}

\newcommand{\bca}{\begin{cases}}
\newcommand{\eca}{\end{cases}}

\newcommand{\bcen}{\begin{center}}
\newcommand{\ecen}{\end{center}}

\newcommand{\bbm}{\begin{bmatrix}}
\newcommand{\ebm}{\end{bmatrix}}

\newcommand{\bmx}{\begin{matrix}}
\newcommand{\emx}{\end{matrix}}

\newcommand{\bpm}{\begin{pmatrix}}
\newcommand{\epm}{\end{pmatrix}}

\newcommand{\btab}{\begin{tabular}}
\newcommand{\etab}{\end{tabular}}

\newcommand{\brn}{\bar{n}}
\newcommand{\brx}{\bar{x}}

\newtheorem{theorem}{Theorem}[section]

\newtheorem{lemma}[theorem]{Lemma}

\newtheorem{assp}[theorem]{Assumption}

\theoremstyle{definition}
\newtheorem{example}[theorem]{Example}

\newtheorem{alg}[theorem]{Algorithm}

\numberwithin{equation}{section}

\theoremstyle{definition}

\begin{document}

\title{Dehomogenization for Completely Positive Tensors}

\author{Jiawang Nie}
\author{Xindong Tang}
\author{Zi Yang}
\author{Suhan Zhong}

\address{Jiawang Nie, Department of Mathematics,
University of California San Diego,
9500 Gilman Drive, La Jolla, CA, USA, 92093.}
\email{njw@math.ucsd.edu}

\address{Xindong Tang, Department of Applied Mathematics,
The Hong Kong Polytechnic University,
Hung Hom, Kowloon, Hong Kong.
}
\email{xindong.tang@polyu.edu.hk}

\address{Zi Yang, Department of Electrical \& Computer Engineering,
University of California, Santa Barbara, CA, USA, 93106.
}
\email{ziy@ucsb.edu}

\address{Suhan Zhong, Department of Mathematics,
Texas A\&M University, College Station, TX, USA, 77843.}
\email{suzhong@tamu.edu}

	


\begin{abstract}
A real symmetric tensor is completely positive (CP)
if it is a sum of  symmetric tensor powers of nonnegative vectors.
We propose a dehomogenization approach for studying CP tensors.
This gives new Moment-SOS relaxations for CP tensors.
Detection for CP tensors and the linear conic optimization
with CP tensor cones can be solved
more efficiently by the dehomogenization approach.
\end{abstract}

\maketitle

\section{Introduction}

Let $\re^n$ be the space of all real $n$-dimensional vectors.
For an integer $d >0$, a tensor $\mc{A}$ of order $d$
over the vector space $\re^n$ is represented by
an array labelled such that
\[
\mc{A} \, = \, ( \mc{A}_{i_1 \ldots i_d}),
\quad  1 \leq i_1, \ldots , i_d  \leq n .
\]
The tensor $\mathcal{A}$ is {\it symmetric}
if all entries $\mathcal{A}_{i_1\ldots i_d}$ are invariant
for all permutations of the label $(i_1, \ldots , i_d)$.
Let $\ttS^d(\mathbb{R}^{n})$ denote the space of
all symmetric tensors of order $d$ over $\re^n$.
For a vector $v \in\mathbb{R}^n$,
$v^{\otimes d}$ denotes the rank-$1$ tensor such that
\[
(v^{\otimes d} )_{i_1,\ldots,i_d} \, = \, v_{i_1}\cdots v_{i_d}
\]
for all labels $i_1, \ldots, i_d$.
Every symmetric tensor is a sum of rank-1 symmetric tensors
\cite{Comon2008}. We refer to \cite{Land12,Lim13}
for introductions to tensors.

Denote by $\re_+^n$ the nonnegative orthant, i.e.,
the set of vectors whose entries are all nonnegative.
A symmetric tensor $\mathcal{A} \in \ttS^d(\re^{n})$ is said to be
{\it completely positive} (CP) if there exist
$v_1, \ldots, v_r \in \re^n_+$ such that
\begin{equation} \label{cdmn}
  \mathcal{A} = (v_1)^{\otimes d} + \cdots + (v_r)^{\otimes d}.
\end{equation}
The equation \reff{cdmn} is called a CP-decomposition, when it exists.
The smallest $r$ in \eqref{cdmn} is the CP-rank of $\mc{A}$,
for which we denote $\rank_{\rm cp}(\mc{A})$.
When \eqref{cdmn} does not exist, we just let
$\rank_{\rm cp}(\mc{A}) = + \infty$.
The cone of all CP tensors in $\mt{S}^d(\re^n)$
is denoted by $\mc{CP}_n^{\otimes d}$.
Generally, it is hard to check whether a tensor is completely positive or not.
The question is NP-hard even for the matrix case (see \cite{Dickinson11}).

Completely positive tensors are closely related to copositive (COP) tensors.
Each $\mc{B} \in \mt{S}^d(\re^n)$ is uniquely determined by
the degree-$d$ homogeneous polynomial
\be \label{df:A(x)}
\mc{B} (x) \, \coloneqq \, \sum_{1\le i_1, \ldots ,i_d \le n}
\mc{B}_{i_1 \ldots  i_d} x_{i_1} \cdots x_{i_d}
\ee
in the variable $x:=(x_1,\ldots, x_n)$.
If $\mc{B}(x) \geq 0$ for every $x \in \re^n$, then
$\mc{B}$ is said to be nonnegative or positive semidefinite.
If $\mc{B}(x) \geq 0$ for every $x \in \re_+^n$, then
$\mc{B}$ is said to be {\it copositive}.
Moreover, if $\mc{B}(x) > 0$ for all $0 \ne x \in \re_+^n$, then
$\mc{B}$ is said to be strictly copositive.
Denote by $\mc{COP}_n^{\otimes d}$ the cone of
all copositive tensors in $\mt{S}^d(\re^n)$.
Memberships for the cone $\mc{COP}_n^{\otimes d}$
can be detected by tight relaxations \cite{Nie19tight,Nie18cop}.

For two tensors $\mc{A}, \mc{B} \in \mt{S}^d(\re^n)$,
their Hilbert-Schmidt inner product is
\be \label{inn:<A,B>}
\langle \mc{A} , \mc{B} \rangle  \, \coloneqq \,
 \sum_{1\le i_1, \ldots ,i_d \le n}
\mc{A}_{i_1 \ldots i_d} \mc{B}_{i_1 \ldots i_d}.
\ee
This induces the {\it Hilbert-Schmidt norm}
\be \label{HSnm:||A||}
 \| \mc{A} \| \, := \, \sqrt{   \langle \mc{A} , \mc{A} \rangle    }.
\ee
If $\mc{A}$ has the CP decomposition as in \reff{cdmn}, then
\[
\langle \mc{A} , \mc{B} \rangle  \, = \,
\mc{B}(v_1) + \cdots + \mc{B}(v_r).
\]
Therefore, if in addition $\mc{B}$ is copositive,
then $\langle \mc{A} , \mc{B} \rangle \ge 0$.
This fact implies the dual relationship
(the superscript $^*$ denotes the dual cone):
\[
\big( \mc{COP}_n^{\otimes d} \big)^* \,=\, \mc{CP}_n^{\otimes d},\quad
\big( \mc{CP}_n^{\otimes d}\big)^* \,=\, \mc{COP}_n^{\otimes d}.
\]

Completely positive tensors are extensions of
completely positive matrices \cite{BermanN,ZhouFan14}.
They have wide applications in
exploratory multi-way data analysis and blind source separation~\cite{Cichocki},
computer vision and statistics~\cite{Shashua}, multi-hypergraphs~\cite{XLQC16},
matrix and tensor completions \cite{ZhouFan14,ZhouFan18}.
We refer to \cite{Fan17sdpCP,LuoQi16,QiXu14NonTenFac}
for recent work on completely positive tensors.

Copositive and CP tensors are important in optimization.
The detection of copositive tensor cones is
a basic optimization problem. We refer to \cite{Nie18cop, Qi13symcop} for recent work.
For CP tensor cones, the memberships are studied in \cite{Fan17sdpCP,QiXu14NonTenFac}.
The CP tensors can be used to construct many interesting problems.
Recent work for CP matrix approximation problem is given in
\cite{FanZhou16CPmat,SpnDur14CPmat}.
Copositive tensors and CP tensors also have various applications.
They are used to formulate interesting models in game theory \cite{Klerk02stanum}
and dynamic systems \cite{Mason07}.
Surveys about copositive and CP optimization can be found
in \cite{Bomzesurvey12, Dursurvey10}.

In this paper, we introduce the dehomogenization approach for
studying completely positive tensors.
This can save computational expenses quite a lot.
In Section~\ref{sec:prelim}, we review some basics for tensor and optimization.
In Section~\ref{sc:RelandMem}, we use the dehomogenization to construct
semidefinite relaxations for CP tensor cones
and apply them to check memberships.
In Section~\ref{sc:LinOptwithCP}, we show how to use the dehomogenization
to solve linear conic optimization with CP cones.
The numerical experiments are given in Section~\ref{sc:ne}.

\section{Preliminaries}
\label{sec:prelim}

\subsection*{Notation}

Let $\re$ (resp.,$\re_+$, $\N$) denote the set of real
(resp., nonnegative real, nonnegative integer) numbers.
The $\re^n$ (resp., $\re_+^n$, $\N^n$) denotes the set of $n$-dimensional vectors
with entries in $\re$ (resp., $\re_+$, $\N$).
For $t\in \re$, $\lceil t\rceil$ is the smallest integer
that is not smaller than $t$.
For an integer $n > 0$, let $[n]:=\{ 1, \ldots, n\}$.
For $x\in\re^n$, $\|x\|$ denotes its Euclidean norm,
the $\delta_x$ denotes the unit mass Dirac measure supported at $x$.
The $e := (1,\ldots,1)$ is the vector of all ones
(its length should be clear in the context), and $e_i$ is the vector
of all zeros except its $i$th entry equal to $1$.
For a vector space $V$, denote by $V^*$ the dual space of $V$,
which is the set of all linear functionals acting on $V$.
For a set $K\subseteq V$, its dual cone is
\[
K^*  \, \coloneqq \,  \{\ell \in V^*: \,
\ell (u)\ge 0\, \forall \, u \in K \}.
\]
A matrix $A\in\re^{n\times n}$ is said to be positive semidefinite (psd)
if $x^TAx\ge 0$ for every $x\in \re^n$.
The inequality $A\succeq 0$ means that $A$ is psd.
The superscript $^T$ denotes the transpose of a vector or matrix.

Let $x=(x_1,\ldots,x_n)$. The $\re[x]$ denotes the ring of polynomials in $x$
with real coefficients. The symbol $\re[x]^{\hm}$
stands for the set of homogeneous polynomials in $\re[x]$.
A homogeneous polynomial is also called a form.
For a degree $d>0$, $\re[x]_d$
denotes the subset of polynomials in $\re[x]$ with degrees up to $d$,
while $\re[x]_d^{\hm}$ stands for the subset of
homogeneous polynomials in $\re[x]$ with degrees equal to $d$.
For a power $\alpha  \coloneqq (\alpha_1,\ldots,\alpha_n)\in \N^n$,
denote the monomial
\[
x^{\alpha} \, \coloneqq \, x_1^{\alpha_1}\cdots x_n^{\alpha_n} .
\]
Define $|\alpha|  \coloneqq  \alpha_1+\cdots+\alpha_n$.
For a degree $d$, denote the power sets
\be
\label{df:Nnd}
\N_d^n = \{\alpha\in \N^n: |\alpha|\le d\}, \quad
\overline{\N}_d^n = \{\alpha\in\N^n: |\alpha|=d\}.
\ee
The vector of all monomials with degrees up to $d$ is denoted as
\[
[x]_d \, \coloneqq \,
\bbm 1 & x_1 & \ldots & x_n & (x_1)^2 & x_1x_2  & \ldots & (x_n)^d \ebm^T.
\]
The subvector of $[x]_d$ with monomials of degree equal to $d$ is denoted as
\be
\label{df:xd_hm}
[x]_d^{\hm} \, \coloneqq \,
(x^{\alpha})_{ \alpha\in \overline{\N}_d^n }.
\ee

\subsection{Copositive forms and CP moments}
\label{ssc:copcp}

Let $\re^{ \N^n_d }$ denote the space of all real vectors
that are labeled by $\alpha\in \N^n_d$.
Each vector in $\re^{ \N^n_d }$ is called a truncated multi-sequences (tms).
Similarly, real vectors labeled by $\alpha\in \overline{\N}^n_d$
are called homogeneous truncated multi-sequence (htms).
The set of all such htms is denoted as $\re^{ \overline{\N}^n_d }$.
A homogeneous truncated multi-sequence
$y \in \re^{ \overline{\N}^n_d }$ is said to be completely positive (CP)
if $y  =  [u_1]_d^\hm + \cdots + [u_r]_d^\hm$
for some nonnegative vectors $u_1, \ldots, u_r \in \re_+^n$.
The set of all CP htms in $\re^{ \overline{\N}^n_d }$
is denoted as $\mc{CP}_{n,d}$. It is a closed convex cone
and can be equivalently written as
\be \label{cone:CP:nd}
\mc{CP}_{n,d} \, = \, \Big\{ \sum\limits_{i=1}^r \lambda_i[u_i]_d^\hm : \,
\lambda_i\ge 0, u_i\in \Delta, r \in \N \Big\} ,
\ee
where $\Dt$ is the simplex
\[
\Delta \, \coloneqq  \, \{ x\in\re^n: e^Tx=1,\, x\ge 0 \}.
\]
A homogeneous polynomial $f$ is said to be copositive if
$f(x)\ge 0$ for every $x\in \re_+^n$.
The cone of copositive forms in $\re[x]_d^\hm$ is denoted as
\be \label{cone:COP:nd}
\mc{COP}_{n,d} \, \coloneqq  \, \{ f \in \re[x]_d^\hm:\, f(x)\ge 0\,(x\in\re_+^n) \}.
\ee
The CP moment cone and COP polynomial cone are dual to each other:
\be
\label{eq:CPCOPdual}
(\mc{CP}_{n,d})^* = \mc{COP}_{n,d},\quad (\mc{COP}_{n,d})^* = \mc{CP}_{n,d}.
\ee

\subsection{Polynomial optimization}
\label{ssc:PO}

Let $\bar{n} = n-1$ and $\bar{x} \coloneqq (x_1,\ldots, x_{n-1})$.
A polynomial $p\in\re[\bar{x}]$ is said to be a sum-of-squares (SOS) if
\[
p = p_1^2+\cdots+p_m^2\quad\mbox{for some}\quad
p_1,\ldots,p_m\in\re[\bar{x}].
\]
The cone of SOS polynomials in $\bar{x}$ is denoted as $\Sigma[\bar{x}]$.
For a degree $d$, denote the truncation
\[
\Sigma[\bar{x}]_d   \coloneqq   \Sigma[\bar{x}]  \cap  \re[\bar{x}]_d .
\]
For a polynomial tuple $g=(g_1,\ldots,g_{m})$, denote the ideal
$\ideal{g} \coloneqq g_1\cdot\re[\bar{x}]+\cdots +g_m\cdot \re[\bar{x}]$
and the quadratic module
$
\qmod{g}  \coloneqq  \Sigma[x] +  g_1 \cdot \Sigma[x] +
{\cdots} + g_{m} \cdot  \Sigma[x]
$.
For a positive integer $k$, the degree-$2k$ truncation of $\ideal{g}$ is denoted as
\[
\ideal{g}_{2k} \coloneqq g_1\cdot\re[\bar{x}]_{2k-\deg(g_1)}+\cdots +g_m\cdot \re[\bar{x}]_{2k-\deg(g_m)}.
\]
Similarly, denote the degree-$2k$ truncation of $\qmod{g}$ as
\[
\qmod{g}_{2k} \,  \coloneqq  \, \Sigma[\bar{x}]_{2k} + g_1\cdot \Sigma[\bar{x}]_{2k-\deg(g_1)}
+{\cdots}+g_{m}\cdot\Sigma[\bar{x}]_{2k-\deg(g_{m})}.
\]
Consider the simplicial set
\[
\overline{ \Dt } = \{  \bar{x} \in \re_+^{n-1}:
x_1 + \cdots + x_{n-1} \le 1 ,\, \Vert \bar{x}\Vert^2 \le 1\}.
\]
In the above, we add a redundant ball constraint $1-\|\bar{x}\|^2\ge 0$ to get tighter relaxations.
Denote by $\mathscr{P}_d( \overline{ \Dt } )$ the cone of polynomials in
$\re[\bar{x}]_d$ that are nonnegative on $\overline{ \Dt }$.
Denote the moment cone
\be
\label{eq:tmscone}
\mathscr{R}_d( \overline{ \Dt } )  \coloneqq  \Big\{
\sum_{i=1}^r \lmd_i [v_i]_{d}:\,  v_i \in \overline{ \Dt },
\lmd_i \ge 0, r \in \N \Big\}.
\ee
The moment cone $\mathscr{R}_d(\overline{ \Dt })$ is closed convex and
\[
(\mathscr{P}_d(\overline{ \Dt }))^* \, = \,  \mathscr{R}_d(\overline{ \Dt }) .
\]

Denote the quadratic module for $\overline{\Dt}$:
\[
\baray{rcr}
\qmod{ \overline{\Dt} }  &  \coloneqq  & \Sig[\bar{x}] +
x_1\cdot \Sig[\bar{x}] + \cdots + x_{n-1} \cdot  \Sig[\bar{x}]  + \\
& & (1-e^T\bar{x}) \Sig[\bar{x}] + (1- \|\bar{x}\|^2) \Sig[\bar{x}].
\earay
\]
Given an even degree $2k >0$, denote the truncation
\[
\baray{rcr}
\qmod{ \overline{\Dt} }_{2k}  & \coloneqq  & \Sig[\bar{x}]_{2k} +
x_1\cdot \Sig[\bar{x}]_{2k-2} + \cdots + x_{n-1} \cdot  \Sig[\bar{x}]_{2k-2}  + \\
& & (1-e^T\bar{x}) \Sig[\bar{x}]_{2k-2} + (1- \|\bar{x}\|^2) \Sig[\bar{x}]_{2k-2}.
\earay
\]
The set $\qmod{ \overline{\Dt} }$ is archimedean,
since it contains the polynomial $1 - \| \bar{x} \|^2$
and $\overline{\Dt}$ is contained in the unit ball.
If a polynomial $p > 0$ on $\overline{\Dt}$, then $p \in \qmod{\overline{\Dt}}$.
Such a conclusion is often referenced as
Putinar's Positivstellensatz \cite{putinar1993positive}.

The dual cones of quadratic modules can be described by localizing matrices.
A tms $z\in \re^{\N_{2k}^{\bar{n}} }$ acts on the polynomial space
$\re[\bar{x}]_{2k}$ as a linear functional such that
\[
\langle \sum_{ \af \in \N^{\bar{n}}_{2k} } p_\af \bar{x}^\af , z \rangle
\, \coloneqq \,
\sum_{ \af \in \N^{\bar{n}}_{2k} } p_\af z_\af  .
\]
For a polynomial $q \in \re[\bar{x}]_{2k}$,
the $k$th order {\it localizing matrix} of $q$ and $z$
is the symmetric matrix $L_{q}^{(k)}[z]$ such that
(the $vec(a)$ denotes the coefficient vector of $a$)
\be \label{df:Lf[y]}
\langle qa^2, y \rangle  \, =  \,
vec(a)^T \big( L_{q}^{(k)}[z]  \big) vec(a)
\ee
for all $a \in \re[x]$ with $\deg(q a^2) \le 2k$.
When $q=1$ (the constant one polynomial),
the localizing matrix $L_{q}^{(k)}[z]$ becomes
the $k$th order moment matrix:
\[
M_k[z]  \coloneqq   L_{1}^{(k)}[z] .
\]
For instance, when $n=3$ and $k=2$, we have $\bar{n}=2$ and
\[
 M_2[z] = \left[
\begin{array}{cccccc}
z_{00} & z_{10} & z_{01} & z_{20} & z_{11} & z_{02} \\
z_{10} & z_{20} & z_{11} & z_{30} & z_{21} & z_{12} \\
z_{01} & z_{11} & z_{02} & z_{21} & z_{12} & z_{03} \\
z_{20} & z_{30} & z_{21} & z_{40} & z_{31} & z_{22} \\
z_{11} & z_{21} & z_{12} & z_{31} & z_{22} & z_{13} \\
z_{02} & z_{12} & z_{03} & z_{22} & z_{13} & z_{04} \\
\end{array} \right],
\]
\[
L_{x_1}^{(2)}[z] = \left[
\begin{array}{cccccc}
z_{10} & z_{20} & z_{11}  \\
z_{20} & z_{30} & z_{21}  \\
z_{11} & z_{21} & z_{12}  \\
\end{array} \right], \quad
L_{x_2}^{(2)}[z] = \left[
\begin{array}{cccccc}
z_{01} & z_{11} & z_{02}   \\
z_{11} & z_{21} & z_{12}   \\
z_{02} & z_{12} & z_{03}   \\
\end{array} \right] ,
\]
\[
L_{1-e^T\bar{x}}^{(2)}[z] = \left[
\begin{array}{cccccc}
z_{00}-z_{10}-z_{01} & z_{10}-z_{20}-z_{11} & z_{01}-z_{11}-z_{02} \\
z_{10}-z_{20}-z_{11} & z_{20}-z_{30}-z_{21} & z_{11}-z_{21}-z_{12} \\
z_{01}-z_{11}-z_{02} & z_{11}-z_{21}-z_{12} & z_{02}-z_{12}-z_{03} \\
\end{array} \right],
\]
\[
L_{1-\| \bar{x}\|^2 }^{(2)}[z] = \left[
\begin{array}{cccccc}
z_{00}-z_{20}-z_{02} & z_{10}-z_{30}-z_{12} & z_{01}-z_{21}-z_{03} \\
z_{10}-z_{30}-z_{12} & z_{20}-z_{40}-z_{22} & z_{11}-z_{31}-z_{13} \\
z_{01}-z_{21}-z_{03} & z_{11}-z_{31}-z_{13} & z_{02}-z_{22}-z_{04} \\
\end{array} \right].
\]
For an even degree $2k$, define the cone
\begin{equation}
\label{def:Sg_2d}
\mathscr{S}[\overline{\Dt}]_{2k} \,  \coloneqq
\left\{ z \in \re^{ \N_{2k}^{\bar{n}} }
\left| \baray{l}
M_k[z]\succeq 0, \, L_{1-e^T\bar{x}}^{(k)}[z] \succeq 0 , \\
L_{x_i}^{(k)}[z]\succeq 0, \, i = 1, \ldots, n-1, \\
L_{1-\| \bar{x}\|^2 }^{(k)}[z]  \succeq 0
\earay \right.
\right\}.
\end{equation}
It can be verified that (see \cite{JNieLinearOptMoment})
\begin{equation}
\label{eq:dualQMnS}
(\qmod{\overline{\Dt}}_{2k})^* \, = \,
\mathscr{S}[\overline{\Dt}]_{2k} .
\end{equation}
We refer to \cite{MSOS20,Las01,Las15,Lau09} for more detailed introductions
to polynomial and moment optimization.
Moment-SOS relaxations are quite useful for solving
matrix and tensor optimization problems.
We refer to \cite{FNZ18,FNZ19,HilNie08,Nie11,NieZhang18,Nie18cop}
for the related work.

\section{Relaxations and memberships}
\label{sc:RelandMem}

In this section, we use the dehomogenization approach to construct
semidefinite relaxations for completely positive tensor cones.
They can be used to check memberships.

Recall the power set $\overline{\N}_d^n$ as in \reff{df:Nnd}.
For each symmetric tensor $\mc{A}\in\ttS^d(\re^n)$,
there exists a unique homogeneous truncated multi-sequence (htms)
$y=(y_{\alpha}) \in \re^{ \overline{\N}_d^n }$ such that
\[
y_{\alpha} = \mc{A}_{i_1,\ldots,i_d}
\]
for every $x^{\alpha} = x_{i_1}\cdots x_{i_d}$
with $\af \in \overline{\N}_d^n$.
This induces the linear map
$\phi:\ttS^d(\re^n) \to \re^{ \overline{\N}_d^n }$ such that
\be
\label{symt2mom:varho}
\phi\Big( \sum_{i=1}^k\lambda_i(u_i)^{\otimes d} \Big) \,=\,
\sum_{i=1}^k \lambda_i [u_i]_d^{\hm},
\ee
for all $u_1,\ldots,u_k \in\re^n$.
The map $\phi$ gives an isomorphism between the symmetric tensor space
$\ttS^d(\re^n)$ and the htms space $\re^{\overline{\N}_d^n}$.
It holds that
\[
\phi(\mc{CP}_n^{\otimes d}) \,=\, \mc{CP}_{n,d},\quad
\phi^{-1}(\mc{CP}_{n,d}) \,=\, \mc{CP}_n^{\otimes d}.
\]
The dual relationship \reff{eq:CPCOPdual} implies that
\[
(\mc{COP}_n^{\otimes d})^* = \mc{CP}_n^{\otimes d},\quad
(\mc{CP}_n^{\otimes d})^* = \mc{COP}_n^{\otimes d}.
\]

\subsection{Dehomogenization for CP moments}
\label{cop:ssc:aprox:dehom}

Let $x=(x_1,\ldots,x_n)$.
Recall that a htms $y\in \re^{\overline{\N}_d^n}$
is completely positive if
\[
y = \lambda_1[u_1]_d^{\hm}+\cdots+\lambda_r[u_r]_d^{\hm}
\]
for scalars $\lambda_1,\ldots,\lambda_r\ge 0$ and points $u_1,\ldots,u_r\in\Dt$, where
\[
\Dt = \{x \in \re^n: x\ge 0, x_1 + \cdots + x_n = 1 \} .
\]
Note the $x_n = 1- (x_1 + \cdots + x_{n-1} )$. Denote
\be \label{hatx:hatn}
\bar{x} \coloneqq  (x_1, \ldots, x_{\bar{n}}), \quad \bar{n} = n-1  .
\ee
This gives the simplicial set
\be \label{simlexbody:Dtarrow}
\overline{\Dt} \, \coloneqq \, \{ \bar{x} \in \re^{\bar{n}} :
\bar{x} \geq 0,
1-e^T\bar{x} \geq 0 \}.
\ee
Here, $e$ is the vector of all ones.
Clearly, we have $x \in \Dt$ if and only if $ \bar{x} \in \overline{\Dt}$.
Define the dehomogenization map
$\varpi: \re[x]_d^{\hm}\rightarrow \re[\bar{x}]_d$ such that
\be \label{diamond:x:f}
\varpi(f) \, \coloneqq \,  f(\bar{x}, 1- e^T\bar{x} ).
\ee
The map $\varpi$ is linear and it gives an isomorphism from
$\re[x]_d^\hm$ to $\re[\bar{x}]_d$. In particular,
the inverse of $\varpi$ is given such that
\[
\varpi^{-1} \Big( \sum_{\alpha\in\N_d^{\bar{n}}}
f_{\alpha}\bar{x}^{\alpha} \Big) \, =  \,
\sum_{\alpha\in \N_d^{\bar{n}}}
f_{\alpha}\bar{x}^{\alpha}(e^Tx)^{d-|\alpha|} .
\]
The adjoint map $\varpi^T$ gives an isomorphism from
$(\re[\bar{x}]_d)^*$ to $(\re[x]_d^{\hm})^*$,
and the inverse adjoint map $\varpi^{-T}$ gives an isomorphism from
$\re^{\overline{\N}_d^n}$ to $\re^{\N_d^{\brn}}$.
For each $y = (y_\af)  \in \re^{ \overline{\N}^n_d }$ and
for each $f \in \re[x]_d^\hm$, it holds that
\be \label{<f,y>=<omg(f),omg-T(y)>}
\langle f, y \rangle  \,= \, \langle \varpi^{-1}(\varpi(f)), y \rangle \,=\,
\langle  \varpi(f),  \varpi^{-T}(y) \rangle .
\ee
Note that a form $f \in \re[x]_d^{\hm}$ is copositive if and only if
$\varpi(f) \in \mathscr{P}_d(\overline{\Dt})$.

The dehomogenization map $\varpi$ can be used to
characterize the CP moment cone $\mc{CP}_{n,d}$.
Recall the truncated moment cone
$\mathscr{R}_d(\overline{\Dt})$ as in \reff{eq:tmscone}.

\begin{lemma} \label{lm:rel:CPnd:dhmg}
A htms $y \in \re^{ \overline{\N}^n_d }$ has the decomposition
\be \label{eq:ydecomp:lem}
y \,=\, \lmd_1 [u_1]_d^\hm + \cdots + \lmd_r [u_r]_d^\hm
\ee
for some $\lambda_i \in \re, u_i \in {\Delta}$ if and only if
\be \label{eq:var(y)decp:lem}
\varpi^{-T}(y) \,=\, \lmd_1 [v_1]_d + \cdots + \lmd_r [v_r]_d,
\ee
where each $v_i$ is the subvector of first $n-1$ entries of $u_i$.
Therefore, a htms $y\in\mc{CP}_{n,d}$ if and only if
$\varpi^{-T}(y)\in\mathscr{R}_d(\overline{\Delta})$.
\end{lemma}
\begin{proof}
Suppose \reff{eq:ydecomp:lem} holds. Let $w = \lmd_1 [v_1]_d + \cdots + \lmd_r [v_r]_d$,
where $v_i$ is the subvector of first $n-1$ entries of $u_i$.
For every $f\in \re[x]_d^{\hm}$ and each $u_i \in \Dt$,
\be \label{eq:lem:con}
\langle f,[u_i]_d^{\hm} \rangle = f(u_i) = \varpi(f)(v_i) = \langle \varpi(f), [v_i]_d \rangle.
\ee
Given the decomposition \reff{eq:ydecomp:lem}, we have
\[ \langle f, y \rangle  \,= \,\sum_{i=1}^r \lambda_i\langle f, [u_i]_d^{\hm} \rangle  \,=\, \sum_{i=1}^r\lambda_i\langle \varpi(f),[v_i]_d\rangle\,=\, \langle \varpi(f),w\rangle. \]
This is true for all $f\in\re[x]_d^{\hm}$, so $\varpi^{-T}(y) = w$,
in view of \reff{<f,y>=<omg(f),omg-T(y)>}.

Conversely, suppose \reff{eq:var(y)decp:lem} holds for
$v_i\in \overline{\Dt}$.
Let $u_i = (v_i,1-e^Tv_i)$ for each $i$. Then, for every $f\in\re[x]_d^{\hm}$,
we have $\varpi(f)(v_i) = f(u_i)$ in view of the definition of
the map $\varpi$. As in \reff{eq:lem:con},  we get
\[
\langle \varpi(f), \varpi^{-T}(y) \rangle = \langle f,\lambda_1[u_1]_d^{\hm}+\cdots+\lambda_r[u_r]_d^{\hm}\rangle,
\]
for all $f \in\re[x]_d^{\hm}$.
The equation \reff{<f,y>=<omg(f),omg-T(y)>} implies
the decomposition \reff{eq:ydecomp:lem}.

Note that $y\in \mc{CP}_{n,d}$ is equivalent to
\[
y \,=\, \lmd_1 [u_1]_d^\hm + \cdots + \lmd_r [u_r]_d^\hm
\]
for some $u_1,\ldots u_r\in \Dt$ and $\lambda_1,\ldots, \lambda_r\in \re_+$.
By the first part, $y \in \mc{CP}_{n,d}$ is equivalent to
$\varpi^{-T}(y) \in \mathscr{R}_d(\overline{\Dt})$.
\end{proof}

\subsection{A relaxation hierarchy}
\label{ssc:relhier}

Lemma~\ref{lm:rel:CPnd:dhmg} implies the equivalence
\be
\label{equiv:CP=Rd}
y \in \mc{CP}_{n,d}  \quad \Longleftrightarrow \quad
\varpi^{-T}(y) \in \mathscr{R}_d(\overline{\Dt}).
\ee
Therefore, approximations for the truncated moment cone $\mathscr{R}_d(\overline{\Dt})$
can be used to approximate $\mc{CP}_{n,d}$.
The simplicial set $\overline{\Dt}$ can be equivalently described as
\be \label{simlexpoly:h:g}
\overline{\Dt} = \{x\in\re^{\bar{n}}:
x_1 \ge 0, \ldots,x_{\bar{n}} \ge 0,
1-e^T\bar{x} \ge 0, 1-\|\bar{x}\|^2 \ge 0 \}.
\ee
The ball constraint $1-\|\bar{x}\|^2\ge 0$ is redundant,
but it can help to get tighter relaxations.
For an order $k \ge 1$, recall the moment cone
$\mathscr{S}[\overline{\Dt}]_{2k}$ defined as in \reff{def:Sg_2d}.
Note that $\mathscr{R}_{2k}(\overline{\Dt}) \subseteq
\mathscr{S}[\overline{\Dt}]_{2k}$ for every $k$.
For each $k \ge d_0 \coloneqq \lceil d/2 \rceil$, define the projection
\be \label{def:Sg:2k:d}
\mathscr{F}[\overline{\Dt}]_k  \coloneqq
\left\{ y \in \re^{\N_d^{\bar{n}}}\Big|
\begin{array}{c}
\exists \, z \in \mathscr{S}[\overline{\Dt}]_{2k}, \\
y_{\alpha} = z_{\alpha}\,(\alpha\in\N_d^{\bar{n}})
\end{array} \right\}.
\ee
It is worthy to note the nesting containment:
\be
\nn \label{nest:RdDt}
\mathscr{F}[\overline{\Dt}]_{d_0} \supseteq \mathscr{F}[\overline{\Dt}]_{d_0+1}
\supseteq \cdots\supseteq \mathscr{R}_d(\overline{\Dt}).
\ee
This motivates the following approximation for $\mc{CP}_{n,d}$:
\be
\label{def:dSk:CPnd}
\HM{\overline{\Dt}}_k  \coloneqq  \{y\in\re^{\overline{\N}_d^n}:
\varpi^{-T}(y)\in\mathscr{F}[\overline{\Dt}]_{k} \}.
\ee
The approximation property of $\HM{\overline{\Dt}}_k$ is given as follows.

\begin{theorem} \label{thm:QMaprx:COCP}
Let $\HM{\overline{\Dt}}_k$ be as above. Then, we have
\be \label{eq:CP=HM}
\bigcap\limits_{k \ge d_0}
\HM{\overline{\Dt}}_k = \mc{CP}_{n,d},
\quad
 \bigcap\limits_{k \ge d_0} \phi^{-1}\big(
\HM{\overline{\Dt}}_k \big) = \mc{CP}_{n}^{\otimes d} .
\ee
\end{theorem}
\begin{proof}
Let $\HM{\overline{\Dt}}$ be the intersection of all $\HM{\overline{\Dt}}_k$.
Since $\mc{CP}_{n,d}\subseteq\HM{\overline{\Dt}}_k$ for all $k\ge d_0$,
so $\mc{CP}_{n,d}\subseteq\HM{\overline{\Dt}}$.
We next prove the reverse containment also holds.
Suppose otherwise there exists $y\in \HM{\overline{\Dt}}$
such that $y\not\in \mc{CP}_{n,d}$.
Then, there must exist a form $f\in \mc{COP}_{n,d}$ such that
\[ \langle f,y \rangle = \langle \varpi(f),\varpi^{-T}(y)\rangle < 0. \]
The existence of the above $f$ is implied by the duality
between copositive and CP tensor cones.
The form $\varpi(f)$ is nonnegative over $\overline{\Dt}$.
For $\varepsilon>0$ sufficiently small, we have
\[
\langle \varpi(f)+\varepsilon,\varpi^{-T}(y) \rangle <0,
\quad  \varpi(f)+\varepsilon>0 \  \text{on} \ \overline{\Dt}.
\]
{Because $\qmod{\overline{\Dt}}$ is archimedean,
the $\varpi(f)+\varepsilon\in \qmod{\overline{\Dt}}_{2k}$
for some $k$}, by Putinar's Positivstellensatz.
Since $\qmod{\overline{\Dt}}_{2k}$ is dual to
$\mathscr{S}[\overline{\Dt}]_{2k}$ and
$\varpi^{-T}(y)\in \mathscr{F}[\overline{\Dt}]_{k}$,
we get the contradiction
\[
\langle \varpi(f)+\varepsilon,\varpi^{-T}(y)\rangle\ge 0 .
\]
Therefore, the first equality in \reff{eq:CP=HM} holds.
The second equality in \reff{eq:CP=HM} follows from the
definition of the linear map $\phi$ as in \reff{symt2mom:varho}.
\end{proof}

\subsection{Checking memberships}
\label{ssc:MemCP}

We discuss how to detect memberships for the CP cones.
For a given tms $y\in\re^{ \bar{\N}^n_d }$,
checking the membership $y\in\mc{CP}_{n,d}$
is equivalent to checking $\varpi^{-T}(y)\in \mathscr{R}_d(\overline{\Dt})$,
by Lemma~\ref{lm:rel:CPnd:dhmg}.
We refer to \cite{Nie14Atkmp}
for how to solve truncated moment problems.

The adjoint inverse map $z =\varpi^{-T}(y)$ can be {shown} as follows.
For each power $\alpha\in\N_d^{\bar{n}}$, it holds that
\[
\varpi^{-1}(\bar{x}^{\alpha})
= \bar{x}^{\alpha}(e^Tx)^{d-|\alpha|}.
\]
For $y\in\re^{\overline{\N}_d^{n}}$, the map $z =\varpi^{-T}(y)$ is given as
\be  \label{eq:varpi:adjinv}
z_{\alpha} = \langle \bar{x}^\af,  \varpi^{-T}(y) \rangle  =
\langle \varpi^{-1} (\bar{x}^\af),  y \rangle  =
\langle  \bar{x}^\af (e^Tx)^{d-|\af|},  y \rangle
\ee
for every $\af  \in \N^{\brn}_d$.
The following is an exposition example.

\begin{example}
\label{ep:ytow}
For the case $n=3,\, d=3$ and
\[
\begin{array}{l}
y = (y_{300},\, y_{210},\, y_{201},\, y_{120},\, y_{111},\, y_{102},\, y_{030},\, y_{021},\, y_{012},\, y_{003}),\\
y_{300} = 3,\, y_{210}=3,\, y_{201}=1,\, y_{120}=2,\, y_{111} = -1,\\
y_{102} = 0,\, y_{030} = 2,\, y_{021} = 2,\, y_{012} = 3,\, y_{003} = 3.
\end{array}
\]
The tms $z =\varpi^{-T}(y)$ is given such that
\[
\begin{array}{l}
z_{00} = y_{300}+3y_{210}+3y_{201}+3y_{120}+6 y_{111}+3y_{102}+y_{030}+3y_{021}\\
\qquad\quad+3y_{012}+ y_{003} = 35,\\
z_{10} = y_{300}+2y_{210}+2y_{021}+y_{120}+2y_{111}+y_{102} = 11 ,\\
z_{01} = y_{210}+2y_{120}+2y_{111}+y_{030}+2y_{021}+y_{012} = 14,\\
z_{20} = y_{300}+y_{210}+y_{201} = 7,\\
z_{11} = y_{210}+y_{120}+y_{111} = 4,\\
z_{02} = y_{120}+y_{030}+y_{021} = 6,\\
z_{30} = y_{300} = 3,\quad z_{21} = y_{210} = 3,\\ z_{12} = y_{120} = 2,\quad z_{03} = y_{030} = 2.
\end{array}
\]
\end{example}

The following is an algorithm for
checking memberships for the cones $\mc{CP}_{n,d}$ and $\mc{CP}_n^{\otimes d}$.

\begin{alg} \label{sdpalg:CPnd}
For the given $\mc{A}\in\ttS^d(\re^n)$,
let $d_0 := \lceil d/2\rceil$, $d_1  :=  \lceil (d+1)/2 \rceil$
and $y = \phi(\mc{A})$.
Do the following:
\begin{description}
\setlength{\itemindent}{.01in}
\item [Step 0]
Set $k=d_1$ and generate a generic polynomial
$R \in \Sigma[\bar{x}]_{2d_1}$.

\item [Step 1]
Solve the semidefinite optimization
\be
\label{eq:Dmom:CP_nd}
\left\{ \baray{rl}
\min & \langle R, z \rangle  \\
\st & z_\af =   \langle  \bar{x}^\af (e^Tx)^{d-|\af|},  y \rangle
\quad  \text{for} \, \,  \af \in \N_d^{\bar{n}} ,  \\
&   z \in\mathscr{S}[\overline{\Dt}]_{2k}.
\earay \right.
\ee
If \reff{eq:Dmom:CP_nd} is infeasible, output that
$y \not\in \mc{CP}_{n,d}$ and $\mc{A}$ is not CP and stop;
otherwise, solve it for a minimizer $z^*$ and let $t \coloneqq d_0$.

\item [Step 2]
Check whether or not the rank condition
\be
\label{eq:rankcondi}
\rank \, M_t[z^*] =  \rank \, M_{t-1}[z^*]
\ee
holds. If it does not hold, go to Step 3;
if it holds, go to Step~4.

\item [Step 3]
If $t<k$, update $t \coloneqq t+1$ and go to Step~2.
Otherwise, update $k \coloneqq k+1$ and go to Step~1.

\item [Step~4]
Let $r  \coloneqq  \rank \, M_t[z^*]$. Compute the $r$-atomic measure
$\mu = \lmd_1 \dt_{v_1} + \cdots + \lmd_r \dt_{v_r}$
for the truncation $z^*|_{2t}$.
Output the CP tensor decomposition
\be  \label{eq:y:CPalgdcomp}
\mc{A} = \lmd_1 \Big[\bpm v_1 \\ 1-e^Tv_1 \epm \Big]^{\otimes d} + \cdots +
\lmd_r \Big[\bpm v_r \\ 1-e^Tv_r \epm \Big]^{\otimes d} .
\ee

\end{description}

\end{alg}

In Step~0, the SOS polynomial $R$ can be chosen as
$R = [\bar{x}]_{d_1}^TR_1^TR_1[\bar{x}]_{d_1}$ for some generic matrix {$R_1$}.
In Step~2, the rank condition \reff{eq:rankcondi} is called the {\it flat truncation},
which is a sufficient and nearly necessary condition for
checking the convergence of moment relaxations; see \cite{Nie13FlatTrunc}.
Once the flat truncation condition is satisfied,
the $r$-atomic measure can be obtained
by using Schur decompositions \cite{HenJas05Glopt}.
In Step~4, the CP decomposition \reff{eq:y:CPalgdcomp} is equivalent to
\[
y = \lmd_1 \Big[\bpm v_1 \\ 1-e^Tv_1 \epm \Big]_d^\hm + \cdots +
\lmd_r \Big[\bpm v_r \\ 1-e^Tv_r \epm \Big]_d^\hm  .
\]
Algorithm~\ref{sdpalg:CPnd} can be implemented with the software
\texttt{GloptiPoly 3} \cite{GloPol3}
and \texttt{SeDuMi} \cite{sturm1999using}.
The asymptotic convergence of Algorithm~\ref{sdpalg:CPnd} is shown as follows.

\begin{theorem}
\label{thm:asymconv}
Let $R$ in \reff{eq:Dmom:CP_nd} be a generic SOS polynomial and let $y = \phi(\mc{A})$.
If $\mc{A}$ is not CP, then the semidefinite relaxation
\reff{eq:Dmom:CP_nd} must be infeasible when $k$ is large enough.
If $\mc{A}$ is CP, then (\ref{eq:Dmom:CP_nd}) has a minimizer for each $k\ge d_1$.
Denote by $z^{*,k}$ the minimizer of (\ref{eq:Dmom:CP_nd}) for the relaxation order $k$.
Then, there exists $t \ge d_1$ such that the truncation sequence
$\{z^{*,k}|_{2t}\}_{k=d_1}^{\infty}$ is bounded
and all its accumulation points satisfy (\ref{eq:rankcondi}).
\end{theorem}
\begin{proof}
When $\mc{A}$ is not CP, we know $y\not\in\HM{\overline{\Dt}}_k$ for all $k$ large enough,
by Theorem \ref{thm:QMaprx:COCP}. This means that
$\varpi^{-T}(y) \not\in \mathscr{S}[\overline{\Dt}]_k$,
so \reff{eq:Dmom:CP_nd} is infeasible for $k$ large enough.
When $\mc{A}$ is CP, we have $y \in\HM{\overline{\Dt}}_k$ for all $k$,
so \reff{eq:Dmom:CP_nd} is feasible for all $k\ge d_1$.
Moreover, since $R$ is a generic SOS polynomial,  the objective of \reff{eq:Dmom:CP_nd}
is coercive in its feasible set and it must have a minimizer $z^{*,k}$.
The convergence conclusion follows from \cite[Theorem~5.3]{Nie14Atkmp}.
\end{proof}

Under some additional assumptions, Algorithm~\ref{sdpalg:CPnd}
has finite convergence.
Consider the following moment optimization problem
\be
\label{eq:Rdmom}
\left\{ \baray{rl}
\min & \langle R, z \rangle  \\
\st & z_\af = \langle  \bar{x}^\af (e^Tx)^{d-|\af|},  y \rangle \,\,
\text{for each} \,\, \af \in \N^{\bar{n}}_d, \\
& z\in\mathscr{R}_{d_1}(\overline{\Dt}).
\earay \right.
\ee
The cone $\mathscr{R}_{d_1}(\overline{\Dt})$ is closed
and its dual cone is $\mathscr{P}_{d_1}(\overline{\Dt})$.
The dual optimization problem of (\ref{eq:Rdmom}) is
\be  \label{eq:Rdpos}
\left\{ \baray{rl}
\max & \langle \rho,\varpi^{-T}(y)\rangle  \\
\st & R-\rho\in \mathscr{P}_{d_1}(\overline{\Dt}), \\
    & \rho \in \re[\bar{x}]_{d}.
\earay \right.
\ee
\begin{theorem}
\label{tm:finconv_cp}
Let $R$ in \reff{eq:Dmom:CP_nd} be a generic SOS polynomial
and let $y = \phi(\mc{A})$.
Suppose $\mc{A}$ is CP.
If (\ref{eq:Rdpos}) has a minimizer $\rho^*$ such that
$R-\rho^*\in\qmod{\overline{\Dt}}$,
then Algorithm~\ref{sdpalg:CPnd} must terminate
within finitely many steps
and output a CP decomposition for $\mc{A}$.
\end{theorem}
\begin{proof}
First notice the following representation
\[
1 - \| \bar{x} \|^2 = (1-e^T\bar{x})(1+ \| \bar{x} \|^2) +
\sum_{i=1}^{ \bar{n} } x_i
\big[ (1-x_i)^2 + \sum_{j\ne i} x_j^2 \big].
\]
This implies that
\[
\qmod{\overline{\Dt}} = \qmod{x_1, \ldots, x_{\bar{n}}, 1-e^T \bar{x} }.
\]
By the given assumption, we have
$R-\rho^* \in \qmod{x_1, \ldots, x_{\bar{n}}, 1-e^T \bar{x} }$.
Suppose $\mc{A}$ is CP, we know $\varpi^{-T}(y)$
admits a representing measure supported in $\overline{\Dt}$.
Consider the following polynomial optimization problem
\be \label{eq:R-rho}
\left\{ \baray{rl}
\min\limits_{\brx\in\re^{\brn}} & R-\rho^*  \\
\st & x_1\ge0,\dots,x_{\brn}\ge0, 1-e^T\brx\ge 0 .
\earay \right.
\ee
By enumerating all possibilities of active constraints,
one can check that for generic $R\in \Sigma[\bar{x}]_{2d_1}$,
the (\ref{eq:R-rho}) only has finitely many complex critical points.
Then, the finite convergence of Algorithm~\ref{sdpalg:CPnd}
follows from \cite[Theorem~5.5]{Nie14Atkmp}.
\end{proof}

In Theorem~\ref{tm:finconv_cp}, the condition $R-\rho^*\in\qmod{\overline{\Dt}}$
is almost necessary and sufficient for the finite convergence to occur.
Indeed, if a polynomial $f$ is nonnegative on $\overline{\Delta}$,
then $f \in\qmod{\overline{\Dt}}$ under some general assumptions
(see \cite{Nie14Optimality}).
The assumption $R-\rho^*\in\qmod{\overline{\Dt}}$ is usually satisfied.
The finite convergence is always observed in our numerical experiments.

\subsection{A comparison with the traditional approach}

\begin{table}[htb!]
\centering
\caption{Comparison of some dimensions}
\label{tab:dim}
\begin{tabular} {|c||c|c||c|c|} \hline
(n, k) & $\binom{n+2k}{2k}$ & $\binom{n-1+2k}{2k}$ & $\binom{n+k}{k}$ & $\binom{n-1+k}{k}$ \\  \hline
(2, 2)  &    15  & 5  & 6  & 3  \\  \hline
(2, 3)  &    28  & 7  &    10  & 4  \\  \hline
(2, 4)  &    45  & 9  &    15  & 5  \\  \hline
(3, 2)  &    35  &    15  &    10  & 6  \\  \hline
(3, 3)  &    84  &    28  &    20  &    10  \\  \hline
(3, 4)  &   165  &    45  &    35  &    15  \\  \hline
(4, 2)  &    70  &    35  &    15  &    10  \\  \hline
(4, 3)  &   210  &    84  &    35  &    20  \\  \hline
(4, 4)  &   495  &   165  &    70  &    35  \\  \hline
(5, 2)  &   126  &    70  &    21  &    15  \\  \hline
(5, 3)  &   462  &   210  &    56  &    35  \\  \hline
(5, 4)  &  1287  &   495  &   126  &    70  \\  \hline
\end{tabular}
\end{table}

To check whether or not $y\in\mc{CP}_{n,d}$,
people can solve moment relaxations directly on $y$.
This requires to solve a hierarchy of semidefinite relaxations as follows:
\be \label{eq:CPnondehom}
\left\{
\begin{array}{cl}
\min & \langle R, z\rangle\\
\st & y = z|_{\overline{\N}_d^n},\, L_{1-e^Tx}^{(k)}[z] = 0,\\
 & z \in \mathscr{S}[x,1-\|x\|^2]_{k},
\end{array}
\right.
\ee
where $R$ is a generic SOS polynomial in $\re[x]_{2d_1}$.
A similar algorithm can be obtained by solving \reff{eq:CPnondehom}.
We refer to \cite{Fan17sdpCP,Nie14Atkmp} for related work.
In the following, we compare the sizes of the relaxations
\reff{eq:Dmom:CP_nd} and \reff{eq:CPnondehom}.
For the relaxation~\reff{eq:CPnondehom},
the dimension of the tms $z \in\re^{\N_{2k}^n}$
is $\binom{n+2k}{2k}$, and the biggest length of
the matrix constraint (the moment matrix $M_k[z]$ has the biggest length)
is $\binom{n+k}{k}$, that is,
\[
\text{dimension of} \,\, z  = \binom{n+2k}{2k},\quad
\text{length of}\,\, M_k[z] = \binom{n+k}{k}.
\]
In comparison, for the relaxation \reff{eq:Dmom:CP_nd},
the corresponding sizes are
\[
\text{dimension of} \,\, z  = \binom{n-1+2k}{2k},\quad
\text{length of}\,\, M_k[z] = \binom{n-1+k}{k}.
\]
For some typical values of $n,k$,
their values are compared in Table~\ref{tab:dim}.
In computational practice,  the relaxation \reff{eq:Dmom:CP_nd}
is more efficient to solve than \reff{eq:CPnondehom}.
In Section~\ref{sc:ne}, we compare performance of
the dehomogenization approach \reff{eq:Dmom:CP_nd}
and the traditional approach \reff{eq:CPnondehom} for detecting
CP tensors in Table~\ref{tab:comparison}.

\section{Linear Conic Optimization with CP Tensors}
\label{sc:LinOptwithCP}

A general linear conic optimization problem
with CP tensors is
\be  \label{org:min<cy>:CPnd}
\left\{
\begin{array}{cl}
\min & c^Tw  \\
\st & \mc{A}_0 + \sum\limits_{i=1}^\ell w_i \mc{A}_i  \in \mc{CP}_n^{\otimes d}, \\
& f_0 + \sum\limits_{i=1}^\ell w_i f_i \in K .
\end{array}
\right.
\ee
In the above, the decision vector is $w \coloneqq (w_1,\ldots, w_\ell)$,
the $\mc{A}_0, \mc{A}_1, \ldots, \mc{A}_\ell$ are given
symmetric tensors in $\ttS^d(\re^n)$,
$K$ is the Cartesian product of some linear, second-order and semidefinite cones,
and $f_0, \ldots, f_\ell$ are given vectors in the {space} of $K$.
This contains a broad class of linear conic optimization
with CP tensors and various constraints.
Its dual optimization problem is
\be  \label{org:max<ap>:COPnd}
\left\{
\begin{array}{cl}
\max & -\langle \mc{A}_0, \mc{X} \rangle - f_0^T \eta \\
\st & \langle \mc{A}_i, \mc{X} \rangle + f_i^T \eta = c_i, \,
  i =1, \ldots, \ell,  \\
&  \mc{X}  \in  \mc{COP}_n^{\otimes d}, \, \eta \in K^* .
\end{array}
\right.
\ee
For each $i = 0, 1, \ldots, \ell$, let
\[
a_i = \phi(\mc{A}_i),
\]
{where $\phi$ is the linear map given as in \eqref{symt2mom:varho}.}
For convenience, denote the linear functions
\begin{equation}
\label{eq:awfw}
a(w) \, \coloneqq \, a_0 + \sum_{i=1}^\ell w_i a_i ,
\quad
f(w) \, \coloneqq \, f_0 + \sum_{i=1}^\ell w_i f_i .
\end{equation}
Then the tensor $\mc{A}(w)$ is CP if and only if $a(w)$
belongs to the CP moment cone $\mc{CP}_{n, d}$.
Therefore, the CP tensor optimization (\ref{org:min<cy>:CPnd}) is equivalent to
\be    \label{min<cy>:CPnd}
\left\{
\begin{array}{cl}
\min  & c^Tw  \\
\st & a_0 + \sum\limits_{i=1}^\ell w_i a_i  \in \mc{CP}_{n, d}, \\
& f_0 + \sum\limits_{i=1}^\ell w_i f_i \in  K .
\end{array}
\right.
\ee
Observe that
\[
  \langle \mc{A}_i, \mc{X} \rangle  \,  =  \,
  \langle a_i, \phi^{-T} ( \mc{X} ) \rangle .
\]
Let $p = \phi^{-T} ( \mc{X} )$, then the dual optimization
\reff{org:max<ap>:COPnd} is equivalent to
\be   \label{max<ap>:COPnd}
\left\{
\begin{array}{cl}
\max & -\langle a_0, p \rangle - f_0^T \eta \\
\st & \langle a_i, p \rangle + f_i^T \eta = c_i, \,
  i =1, \ldots, \ell,  \\
&  p \in  \mc{COP}_{n, d}, \, \eta \in K^* .
\end{array}
\right.
\ee
Recall the dehomogenization linear map $\varpi$ is given as in \reff{diamond:x:f}.
By Lemma~\ref{lm:rel:CPnd:dhmg}, we have $a(w) \in \mc{CP}_{n, d}$ if and only if
$\varpi^{-T}( a(w) ) \in \mathscr{R}_d( \overline{\Dt} ).$
For each $i = 0, 1, \ldots, \ell$, let
\[
\hat{a}_i  \, = \, \varpi^{-T}( a_i ) .
\]
Therefore, the optimization \reff{min<cy>:CPnd} is equivalent to
\be  \label{min<cy>:CPnd:varpi}
\left\{
\begin{array}{cl}
\min  & c^T w  \\
\st & \hat{a}_0 + \sum\limits_{i=1}^\ell w_i \hat{a}_i
\in \mathscr{R}_d( \overline{\Dt}  ), \\
& f_0 + \sum\limits_{i=1}^\ell w_i f_i  \in  K .
\end{array}
\right.
\ee
For convenience, we denote that
\begin{equation}
\label{eq:hataw}
\hat{a}(w)  \, \coloneqq  \, \hat{a}_0 + \sum_{i=1}^\ell w_i \hat{a}_i  .
\end{equation}
Observe that
\[
\langle a_i, p \rangle  \, = \, \langle \varpi^{-T}( a_i ),  \varpi(p) \rangle .
\]
Let $\hat{p}  = \varpi(p)$.
The dual optimization \reff{max<ap>:COPnd} is equivalent to
\be  \label{max:COP:varpi}
\left\{
\begin{array}{cl}
\max & -\langle \hat{a}_0, \hat{p}   \rangle - f_0^T \eta \\
\st & \langle \hat{a}_i, \hat{p}  \rangle + f_i^T \eta = c_i, \,
  i =1, \ldots, \ell,  \\
&  \hat{p} \in \mathscr{P}_d( \overline{\Dt}  ), \, \eta \in K^* .
\end{array}
\right.
\ee

The Moment-SOS hierarchy can be used to solve them.
The $k$th order moment relaxation for \reff{min<cy>:CPnd:varpi}  is
\be   \label{min<cw>:Lg[w]psd}
\left\{
\begin{array}{cl}
\min  & c^T w   \\
\st & \hat{a}_0 + \sum\limits_{i=1}^\ell w_i \hat{a}_i = z|_d , \\
&  L_{g_i}^{(k)}[z] \succeq 0, i=1,\ldots, n+1, \\
& f_0 + \sum\limits_{i=1}^\ell w_i f_i  \in  K , \\
& M_k[z] \succeq 0, z \in \re^{ \N_{2k}^{n-1}  }.
\end{array}
\right.
\ee
In the above, $(g_1, \ldots, g_n, g_{n+1})  =
(x_1, \ldots, x_{n-1}, 1 -e^T\bar{x}, 1 - \| \bar{x} \|^2)$.
Its dual problem is the $k$th order SOS relaxation
\be   \label{max:QMSOS:varpi}
\left\{
\begin{array}{cl}
\max & -\langle \hat{a}_0, \hat{p} \rangle - f_0^T \eta \\
\st & \langle \hat{a}_i, \hat{p}  \rangle + f_i^T \eta = c_i, \,
  i =1, \ldots, \ell,  \\
&  \hat{p} \in \qmod{\overline{\Dt}}_{2k}, \, \eta \in K^* .
\end{array}
\right.
\ee

The following is the Moment-SOS algorithm for
solving the linear conic optimization \reff{org:min<cy>:CPnd}.

\begin{alg}
\label{alg:linconc:CP}
Let $d_0  \coloneqq  \lceil d/2 \rceil$ and $k \coloneqq d_0$.
Do the following:

\begin{description}
\setlength{\itemindent}{.01in}
\item [Step~1]
Solve the $k$th order Moment-SOS relaxation pair
\reff{min<cw>:Lg[w]psd}-\reff{max:QMSOS:varpi}.
If \reff{min<cw>:Lg[w]psd} is infeasible,
stop and output that \reff{org:min<cy>:CPnd} is infeasible.
If \reff{max:QMSOS:varpi} is unbounded from above,
stop and output that \reff{org:max<ap>:COPnd} is unbounded above.
Otherwise, compute the optimizers $(w^{(k)}, z^{(k)})$,
$(\hat{p}^{(k)}, \eta^{(k)})$ for
\reff{min<cw>:Lg[w]psd}-\reff{max:QMSOS:varpi} respectively.

\item [Step~2]
For $t\in [d_0,k]$, check if the rank truncation
\be  \label{flat:rank}
 \rank \, M_t[z^{(k)}] \, = \, \rank \, M_{t-1}[z^{(k)}]
\ee
holds or not. If it does, then $w^{(k)}$ is a minimizer of \reff{org:min<cy>:CPnd}.
Moreover, if
\[
c^T w^{(k)} + \langle \hat{a}_0, \hat{p}^{(k)}\rangle + f_0^T\eta^{(k)} = 0,
\]
then $(\hat{p}^{(k)},\eta^{(k)})$ is a maximizer of \reff{max:COP:varpi},
or equivalently,
\[ (\phi^T( \varpi^{-1}( \hat{p}^{(k)} )), \, \eta^{(k)}) \]
is a maximizer of \reff{org:max<ap>:COPnd}.
Otherwise, let $k \coloneqq  k+1$ and go to Step~1.

\end{description}

\end{alg}

Algorithm~\ref{alg:linconc:CP} can be implemented with the software
\texttt{GloptiPoly 3} \cite{GloPol3} and \texttt{SeDuMi} \cite{sturm1999using}.
The \reff{min<cw>:Lg[w]psd} is a relaxation of \reff{min<cy>:CPnd:varpi},
which is equivalent to the CP tensor optimization \reff{org:min<cy>:CPnd}.
If \reff{min<cw>:Lg[w]psd} is infeasible, then
\reff{min<cy>:CPnd:varpi} is also infeasible.
Suppose $w^*$ is the minimizer of \reff{min<cy>:CPnd:varpi}.
If the rank condition \reff{flat:rank} is met,
then $\hat{a}(w^{(k)})  \in \mathscr{R}_d(\overline{\Dt})$
and it is feasible for \reff{min<cy>:CPnd:varpi}. For such a case, \reff{min<cw>:Lg[w]psd}
is a tight relaxation of \reff{min<cy>:CPnd:varpi}
and $w^{(k)}$ is a minimizer of \reff{min<cy>:CPnd:varpi}.
Moreover, if in addition the maximizer $(\hat{p}^{(k)},\eta^{(k)})$
of \reff{max:QMSOS:varpi} satisfies the equation
$
c^Tw^{(k)}  + \langle \hat{a}_0, \hat{p}^{(k)}\rangle + f_0^T\eta^{(k)} = 0,
$
then $(\hat{p}^{(k)},\eta^{(k)})$ is the maximizer of \reff{max:COP:varpi}.
The asymptotic convergence of Algorithm~\ref{alg:linconc:CP}
is shown as follows.

\begin{theorem}
Assume the vectors $\hat{a}_0, \ldots, \hat{a}_\ell$ are linearly independent.
Suppose \reff{min<cy>:CPnd:varpi} is feasible
and \reff{max:COP:varpi} has a feasible $\hat{p}$
that is positive on $\overline{\Dt}$.
Let $(w^{(k)}, z^{(k)})$ be a minimizer for the moment relaxation
\reff{min<cw>:Lg[w]psd} with the relaxation order $k$.
Then, the sequence $\{w^{(k)} \}_{k=d_0}^\infty$ is bounded and
each accumulation point of  $\{w^{(k)} \}_{k=d_0}^\infty$
is a minimizer for \reff{min<cy>:CPnd:varpi}.
\end{theorem}
\begin{proof}
Suppose $(\hat{p}_1, \eta_1)$ is a feasible pair of \reff{max:COP:varpi}
such that $\hat{p}_1 > \eps_0 $ on $\overline{\Dt}$,
for some $\eps_0 >0$.
First, we show that $\{ z^{(k)}|_d \}_{k=d_0}^\infty$ is a bounded sequence.
For each $k$, we have
\[
0 \leq \langle \hat{p}_1 - \eps_0, z^{(k)} \rangle  +
  f(w^{(k)})^T \eta_1   =
 \langle \hat{p}_1 , z^{(k)} \rangle +  f(w^{(k)})^T \eta_1
 - \eps_0 \langle 1, z^{(k)} \rangle ,
\]
 where $f(w)$ is given as in \reff{eq:awfw}.
The first equality constraint in \reff{min<cw>:Lg[w]psd} implies that
\[
\langle \hat{p}_1 , z^{(k)} \rangle  \, = \,
\langle \hat{p}_1, \hat{a}_0  \rangle  +
\sum_{i=1}^l (w^{(k)})_i \langle \hat{p}_1,   \hat{a}_i  \rangle .
\]
Also observe that
\[
f(w^{(k)})^T \eta_1 \, = \,  f_0^T \eta_1 +
\sum_{i=1}^l   (w^{(k)})_i f_i^T \eta_1 \, = \,
f_0^T \eta_1 +
\sum_{i=1}^l   (w^{(k)})_i (c_i - \langle \hat{p}_1, \hat{a}_i \rangle ).
\]
Therefore, we get that
\[
0 \le    c^T w^{(k)}  + \langle \hat{p}_1, \hat{a}_0\rangle + f_0^T \eta_1
 - \eps_0 \langle 1, z^{(k)} \rangle,
\]
\[
 \eps_0 \langle 1, z^{(k)} \rangle \le
c^T w^{(k)}  + \langle \hat{p}_1, \hat{a}_0 \rangle + f_0^T \eta_1  .
\]
Let $w^*$ be a feasible point of \reff{min<cy>:CPnd:varpi},
then $c^T w^{(k)} \le c^T w^*$,
since \reff{min<cw>:Lg[w]psd} is a relaxation of \reff{min<cy>:CPnd:varpi}.
Hence,
\[
(z^{(k)})_0 \leq \frac{c^T w^* + \langle \hat{p}_1, \hat{a}_0 \rangle + f_0^T \eta_1 }{\eps_0}.
\]
This implies the sequence $\{ (z^{(k)})_0 \}_{k=d_0}^\infty$ is bounded.
Moreover, the localizing matrix inequality
\[
L_{1-\| \bar{x} \|^2}^{(k)}[z^{(k)}] \succeq 0
\]
implies that
\[
\langle (1-\| \bar{x} \|^2) \bar{x}^{2\af},  z^{(k)}  \rangle \ge 0
\]
for every square $\bar{x}^{2\af}$ with
$|\af| = 0, 1, \ldots, k-1$.
This shows that the diagonals of $M_k[z^{(k)}]$
are bounded by a constant multiple of $(z^{(k)})_0$.
This shows that the sequence $\{z^{(k)}|_d\}_{k=d_0}^\infty$ is bounded.
Note $z^{(k)}$ satisfies the linear equality constraint
\[
\hat{a}_0 + \sum_{i=1}^\ell w_i^{(k)} \hat{a}_i = z^{(k)}|_d .
\]
Since $\hat{a}_0, \ldots, \hat{a}_\ell$ are linearly independent,
we also have that $\{w^{(k)}\}_{k=d_0}^\infty$ is a bounded sequence.

Let $z^{\infty}$ be an accumulation point of $\{z^{(k)}|_d\}$.
This also gives an accumulation point $w^{\infty}$ of $\{ w^{(k)} \}$.
One can show that $(w^{\infty}, z^{\infty})$
is a minimizer of \reff{min<cy>:CPnd:varpi}.
The proof is the same as for the one in
\cite[Theorem~4.3]{JNieLinearOptMoment}.
We refer to \cite{JNieLinearOptMoment} for more details.
\end{proof}

Under some conditions,
Algorithm~\ref{alg:linconc:CP} terminates within finitely many loops.
We need the following assumption.

\begin{assp}
\label{as:finKKT}
The optimization \reff{max:COP:varpi} has a maximizer pair $(\hat{p}^*,\eta^*)$
such that $\hat{p}^* \in  \qmod{\bar{x}, 1-e^T \bar{x}, 1-\|\bar{x}\|^2}$
and the optimization problem
\be  \label{eq:f-G-F}
\left\{ \baray{rl}
\min\limits_{\brx\in\re^{\brn}} & \hat{p}^*(\brx)\\
\st &  x_1\ge0, \ldots, x_{\brn}\ge0, \, 1-e^T\brx\ge 0 .
\earay \right.
\ee
has finitely many critical points on which
the objective value equals $0$.
\end{assp}

The following theorem shows the finite termination of
Algorithm~\ref{alg:linconc:CP}.

\begin{theorem}\label{tm:linfinconv}
Let $w^*$ and $(\hat{p}^*, \eta^*)$ be optimizers for
\reff{min<cy>:CPnd:varpi} and \reff{max:COP:varpi} respectively.
Suppose \reff{min<cy>:CPnd:varpi} and \reff{max:COP:varpi}
have the same optimal value. If Assumption~\ref{as:finKKT} holds, then
Algorithm~\ref{alg:linconc:CP} terminates within finitely many steps
and returns optimizers for \reff{min<cy>:CPnd:varpi} and \reff{max:COP:varpi}.
\end{theorem}
\begin{proof}
By the zero duality gap assumption,
\[
0 = c^T w^* + \langle \hat{a}_0, \hat{p}^* \rangle +
f_0^T \eta^* = \langle \hat{a}(w^*), \hat{p}^* \rangle + f(w^*)^T \eta^*,
\]
where $\hat{a}(w)$ is given as in \reff{eq:hataw}
and $f(w)$ is as in \reff{eq:awfw}.
In the above, the second equality is from
the equality constraints in \reff{max:COP:varpi}.
This implies that
\[
\langle \hat{a}(w^*), \hat{p}^* \rangle = f(w^*)^T \eta^* = 0.
\]
Let $\mu^*$ be a representing measure for $\hat{a}(w^*)$,
which is supported in the set $\overline{\Dt}$. Then, the polynomial
$\hat{p}^*$ vanishes identically on the support $\supp{\mu^*}$.
Each point of $\supp{\mu^*}$ is a minimizer of \reff{eq:f-G-F},
whose minimum value is $0$. Also note that
\[
\qmod{\bar{x}, 1-e^T \bar{x}, 1-\|\bar{x}\|^2}
\, =  \,  \qmod{\bar{x}, 1-e^T \bar{x} }.
\]
The $k$th order SOS relaxation for \reff{eq:f-G-F} is
\be \label{sos:minc*}
\left\{ \baray{rl}
\gamma_k := \max & \gamma  \\
\st & \hat{p}^*-\gamma \in  \qmod{\bar{x}, 1-e^T \bar{x} }_{2k} .
\earay \right.
\ee
The Assumption~\ref{as:finKKT} implies that the hierarchy
of \reff{sos:minc*} has finite convergence, i.e.,
$\gamma_k=0$ for all $k\geq N_1$, for some order $N_1$.
The relaxation \reff{sos:minc*}
achieves its optimal value for all $k$ high enough,
by Assumption~\ref{as:finKKT}.
The dual optimization problem of \reff{sos:minc*}
is the moment relaxation
\be \label{mom:minc*}
\left\{\baray{rl}
\min\limits_{ z }  & \langle \hat{p}^*, z \rangle \\
\st & L_{g_i}^{(k)}[z] \succeq 0, i=1,\ldots, n, \\
& z_0 = 1,  M_k[z] \succeq 0, z \in \re^{ \N_{2k}^{\bar{n}}  }.
\earay  \right.
\ee
In the above, $(g_1, \ldots, g_n)  =
(x_1, \ldots, x_{n-1}, 1 -e^T\bar{x})$.
By Assumption~\ref{as:finKKT}, the optimization problem
\reff{eq:f-G-F} has only finitely many critical points with the objective value $0$.
So the Assumption~2.1 in \cite{Nie13FlatTrunc}
for the problem \reff{eq:f-G-F} is met.

Suppose $z^{(k)}$ is optimal for \reff{min<cw>:Lg[w]psd}.
If $(z^{(k)})_0 = 0$, then $z^{(k)} =0$,
because of the constraints $M_k[z^{(k)}] \succeq 0$
and each $L_{g_i}^{(k)}[z^{(k)}] \succeq 0$
(see Lemma~5.7] of \cite{Lau09}).
This means that $z^{(k)}$ is identically zero and
\reff{flat:rank} must hold.
If $(z^{(k)})_0 >0$, then we can scale $z^{(k)}$
such that $(w^{(k)})_0=1$.
Furthermore, the scaled $z^{(k)}$ is a minimizer of \reff{mom:minc*}.
By Theorem~2.2 of \cite{Nie13FlatTrunc},
the terminating criterion \reff{flat:rank} must hold
when $k$ is sufficiently large.
\end{proof}

\section{Numerical experiments}
\label{sc:ne}

This section gives numerical experiments for using the dehomogenization method,
especially for Algorithms \ref{sdpalg:CPnd} and \ref{alg:linconc:CP}.
The numerical examples are solved with \texttt{MATLAB}
software {\tt GloptiPoly3} \cite{GloPol3}
and {\tt SeDuMi} \cite{sturm1999using}.
The computation is implemented in MATLAB R2018a,
in a Laptop with CPU 8th Generation Intel® Core™ i5-8250U and RAM 16 GB.

\subsection{CP tensors}

First, we show how to use the dehomogenization method to check memberships of CP tensors.
Given a symmetric tensor $\mathcal{A}\in\ttS^d(\re^n)$, we first compute the htms
$y=\phi(\mathcal{A})$ and then use Algorithm~\ref{sdpalg:CPnd}
to test whether or not $y\in\mc{CP}_{n,d}$. It is worth to note that
$y\in \mc{CP}_{n,d}$ is equivalent to that $\mc{A}\in \mc{CP}_n^{\otimes d}$.
If $\lambda_1[u_1]_d^{\hm}+\cdots+ \lambda_r[u_r]_d^{\hm}$
is a computed CP decomposition for $y$,
the decomposition accuracy is measured as the Euclidean norm
\[
\| \lambda_1[u_1]_d^{\hm}+\cdots+ \lambda_r[u_r]_d^{\hm} - y\|.
\]
We begin with some explicit examples.

\begin{example}
\label{ex:CPmat}
Consider the symmetric matrices
\[
A = \bbm
6   &  4   &  1  &   2  &   2\\
4   &  5   &  0   &  1   &  3\\
1   &  0   &  3   &  1   &  2\\
2   &  1   &  1   &  1   &  1\\
2   &  3   &  2   &  1   &  5
\ebm,\quad
B = \bbm
2 & 1 & 0 & 0 & 0\\
1 & 2 & 1 & 0 & 0\\
0 & 1 & 2 & 2 & 2\\
0 & 0 & 2 & 3 & 3\\
0 & 0 & 2 & 3 & 4
\ebm,\quad
C = \bbm
1 & 1 & 2 & 3 & 4\\
1 & 1 & 3 & 2 & 3\\
2 & 3 & 3 & 3 & 3\\
3 & 2 & 3 & 1 & 4\\
4 & 3 & 3 & 4 & 5
\ebm.
\]
Denote $y_A=\phi(A), y_B = \phi(B), y_C = \phi(C)$. It is easy to compute that
\[
\begin{array}{l}
\varpi^{-T}(y_A) = (54, 15, 13, 7, 6, 6, 4, 1, 2, 5, 0, 1, 3, 1, 1),\\
\varpi^{-T}(y_B) = (31, 3, 4, 7, 8, 2, 1, 0, 0, 2, 1, 0, 2, 2, 3),\\
\varpi^{-T}(y_C) = (67, 11, 10, 14, 13, 1, 1, 2, 3, 1, 3, 2, 3, 3, 1).
\end{array}
\]
For matrix $A$, Algorithm~\ref{sdpalg:CPnd} terminates at relaxation
$k=3$ within $1.03$ seconds.
The $A$ is CP with the output decomposition
$A = \sum_{i=1}^6\lambda_i (u_i)^{\otimes 2}$ (accuracy is $1.3879\cdot 10^{-6}$),
\[  \begin{array}{ll}
u_1 = (0, 0.3883, 0, 0, 0.6117)), & \lambda_1 = 8.4610;\\
u_2 = (0, 0, 0.5070, 0, 0.4930), & \lambda_2 = 3.1782;\\
u_3 = (0.4701, 0, 0, 0.3270, 0.2029), & \lambda_3 = 2.4437;\\
u_4 = (0, 0, 0.5398, 0.2150, 0.2452), & \lambda_4 = 4.5418;\\
u_5 = (0.3405, 0, 0.2927, 0.1610, 0.2058), & \lambda_5 = 10.0351;\\
u_6 = (0.4118, 0.3834, 0, 0.1029, 0.1019), & \lambda_6 = 25.3401.
\end{array}  \]
For matrix $B$, Algorithm~\ref{sdpalg:CPnd} terminates at relaxation $k=2$ within $0.52$ second.
The $B$ is CP with the output decomposition
$A = \sum_{i=1}^4\lambda_i(u_i)^{\otimes 2}$ (accuracy is $1.9780\cdot 10^{-6}$),
\[ \begin{array}{ll}
u_1 = (0, 0, 0, 0, 1), & \lambda_1 = 1;\\
u_2 = \frac{1}{8}(2, 0, 0, 3, 3), & \lambda_2 = \frac{64}{3};\\
u_3 = \frac{1}{3}(2, 1, 0, 0, 0), & \lambda_3 = \frac{9}{2};\\
u_4 = \frac{1}{5}(0, 3, 0, 2, 0), & \lambda_4 = \frac{25}{6}.
\end{array} \]
For matrix $C$, Algorithm~\ref{sdpalg:CPnd} terminates at relaxation $k=2$ within $0.11$ second.
Since the moment relaxation is infeasible, $C$ is not CP.
\end{example}

Then we apply Algorithm \ref{sdpalg:CPnd} to tensors with larger dimensions and higher orders.
\begin{example}
\label{ex:CPmemtensor}
(i) Consider the symmetric tensor
\[ \mc{A} = 3\bbm 0\\1\\0\ebm^{\otimes 6}+\bbm -1\\3\\1\ebm^{\otimes 6}+
3\bbm 1\\2\\2\ebm^{\otimes 6}+2\bbm 2\\3\\2\ebm^{\otimes 6}. \]
By applying Algorithm~\ref{sdpalg:CPnd}, we detect that $\mc{A}$
is not CP at the relaxation order $k=4$. It took $0.22$ second.

(ii) Consider the symmetric tensor
\[ \mc{A} = \frac{1}{100}\Big(
7\bbm 0\\1\\1\\0\ebm^{\otimes 4}+5\bbm 0\\2\\1\\0\ebm^{\otimes 4}+
6\bbm 0\\0\\2\\2\ebm^{\otimes 4}
+7\bbm 1\\2\\1\\1\ebm^{\otimes 4}+6\bbm 1\\2\\0\\0\ebm^{\otimes 4} \Big).
\]
By applying Algorithm~\ref{sdpalg:CPnd}, we detect that $\mc{A}$
is CP at the relaxation order $k=3$.
It took $0.49$ second. The output CP decomposition is
$\mc{A} = \sum_{i=1}^5\lambda_i (u_i)^{\otimes 4}$ (accuracy is $4.1353\cdot10^{-6}$),
\[ \begin{array}{ll}
u_1 = \frac{1}{2}(0, 0, 1, 1), & \lambda_1 = 15.36;\\
u_2 = \frac{1}{2}(0, 1, 1, 0), & \lambda_2 = 1.12;\\
u_3 = \frac{1}{5}(1, 2, 1, 1), & \lambda_3 = 43.75;\\
u_4 = \frac{1}{3}(0, 2, 1, 0), & \lambda_4 = 4.05;\\
u_5 = \frac{1}{3}(1, 2, 0, 0), & \lambda_5 = 4.86.
\end{array}
 \]
\end{example}

\begin{example}
\label{ex:CPtensor}
(i)
Consider the symmetric tensor $\mc{A}\in \ttS^3(\re^5)$ such that
\[ \begin{array}{cl}
\phi(\mc{A})\  = &( 4, 2, 3, 1, 4,  2, 2, 0, 2, 3, 0, 3, 1, 1, 4, 5, 4, 3, \\
 & \qquad \qquad \qquad  3, 4, 2, 3, 3, 1, 3,  6, 2, 4, 2, 1, 4, 6, 4, 4, 7).
\end{array}\]
By applying Algorithm~\ref{sdpalg:CPnd}, we detect that $\mc{A}$
is CP at the relaxation order $k=3$.
It took $2.22$ seconds.
The output CP decomposition is $\mc{A} =
\sum_{i=1}^7\lambda_i (u_i)^{\otimes 3}$ (accuracy is $4.9617\cdot10^{-6}$) with
\[ \begin{array}{ll}
u_1 = \frac{1}{3}(1, 0, 1, 0, 1), & \lambda_1 = 27;\\
u_2 = \frac{1}{2}(0, 0, 0, 1, 1), & \lambda_2 = 16;\\
u_3 = \frac{1}{3}(1, 0, 0, 1, 1), & \lambda_3 = 27;\\
u_4 = \frac{1}{4}(1, 1, 1, 0, 1), & \lambda_4 = 128;\\
u_5 = \frac{1}{4}(0, 1, 1, 1, 1), & \lambda_5 = 64;\\
u_6 = (0, 0.2165, 0.5669, 0.2165, 0), & \lambda_6 = 10.3974;\\
u_7 = (0, 0.4198, 0.1604, 0.4198, 0), & \lambda_7 =  25.6026.\\
\end{array}
\]

{(ii)
Consider the symmetric tensor $\mc{A}\in \ttS^6(\re^4)$ such that
\[ \begin{array}{cl}
\phi(\mc{A})\  = &(
3, 3, 4, 3, 3, 4, 3, 6, 4, 5, 3, 4, 3, 6, 4, 5, 10, 6, 6, 9, 3,4, 3,6, 4,\\
&\quad 5, 10, 6, 6, 9, 18, 10, 8, 10, 17, 3, 4, 3, 6, 4, 5, 10, 6, 6, 9, 18, 10,\\
&\quad \quad 8, 10, 17, 34, 18, 12, 12, 18, 33, 9, 8, 6, 16, 5, 8, 38,7, 7, 12, 100, \\
&\quad \quad \quad 11, 9, 11, 20, 278, 19, 13, 13, 19, 36, 797, 36, 22, 18, 22, 36, 69).
\end{array}\]
By applying Algorithm~\ref{sdpalg:CPnd}, we detect that $\mc{A}$
is CP at the relaxation order $k=3$.
It took $1.48$ seconds.
The output CP decomposition is $\mc{A} =
\sum_{i=1}^7\lambda_i (u_i)^{\otimes 6}$ (accuracy is $9.1718\cdot10^{-8}$) with
\[ \begin{array}{ll}
u_1 = \frac{1}{4}(0, 1, 3, 0), & \lambda_1 = 4096;\\
u_2 = \frac{1}{2}(0, 0, 1, 1), & \lambda_2 = 64;\\
u_3 = \frac{1}{5}(1, 1, 2, 1), & \lambda_3 = 15625;\\
u_4 = \frac{1}{3}(0, 1, 1, 1), & \lambda_4 = 729;\\
u_5 = \frac{1}{3}(1, 1, 1, 0), & \lambda_5 = 729;\\
u_6 = \frac{1}{4}(1, 1, 1, 2), & \lambda_6 =15625;\\
u_7 = (0, 1, 0, 0), & \lambda_7 =  2;\\
u_8 = \frac{1}{2}(0, 1, 0, 1), & \lambda_8=128.
\end{array}
\]
}
\end{example}

{
\begin{example}\label{ex:large}
Consider the symmetric tensor
\[\begin{array}{l} \mc{A} = \frac{1}{100}\Big(
\bbm 0\\1\\0\\1\ebm^{\otimes 10}
+\bbm 1\\1\\2\\1\ebm^{\otimes 10}
+\bbm 0\\1\\1\\1\ebm^{\otimes 10}
+\bbm 1\\2\\1\\0\ebm^{\otimes 10}
+\bbm 0\\1\\1\\0\ebm^{\otimes 10} \\
\qquad\qquad\quad
+\bbm 1\\1\\0\\1\ebm^{\otimes 10}
+\bbm 0\\1\\0\\1\ebm^{\otimes 10}
+\bbm 2\\1\\0\\2\ebm^{\otimes 10}
+\bbm 1\\0\\1\\1\ebm^{\otimes 10}
+\bbm 1\\1\\1\\2\ebm^{\otimes 10} \Big).
\end{array} \]
By applying Algorithm~\ref{sdpalg:CPnd}, we detect that
$\mc{A}$ is CP at the relaxation order $k=6$.
It took $3.77$ seconds. The output CP decomposition is
$\mc{A} = \sum_{i=1}^9\lambda_i (u_i)^{\otimes 10}$
(accuracy is $1.0654\cdot 10^{-9}$),
\[ \begin{array}{ll}
u_1 = \frac{1}{5}(2, 1, 0, 2), & \lambda_1 = 97656.2507;\\
u_2 = \frac{1}{3}(1, 0, 1, 1), & \lambda_2 = 590.4877;\\
u_3 = \frac{1}{3}(1, 1, 0, 1), & \lambda_3 = 590.4913;\\
u_4 = \frac{1}{5}(1, 1, 1, 2), & \lambda_4 = 97656.25;\\
u_5 = \frac{1}{2}(0, 1, 0, 1), & \lambda_5 = 20.4827;\\
u_6 = \frac{1}{3}(0, 1, 1, 1), & \lambda_6 = 590.4894;\\
u_7 = \frac{1}{5}(1, 1, 2, 1), & \lambda_7 = 97656.2485;\\
u_8 = \frac{1}{4}(1, 2, 1, 0), & \lambda_8 = 10485.7611;\\
u_9 = \frac{1}{2}(0, 1, 1, 0), & \lambda_9 = 10.2293.\\
\end{array}
 \]
\end{example}
In Table~\ref{tab:comparison},
we compare numerical performance of our dehomogenization approach, i.e., the relaxation (\ref{eq:Dmom:CP_nd}), with the traditional approach, i.e.,
the relaxation (\ref{eq:CPnondehom}).
In the table, the first column represents examples in this subsection,
and \ref{ex:CPmat}(A), \ref{ex:CPmat}(B), \ref{ex:CPmat}(C) are the CP matrix problems in Example~\ref{ex:CPmat} given by matrices $A$, $B$ and $C$ respectively.
Time consumption in the table is measured in seconds.
The order represents the relaxation order for getting the accuracy.
It is clear from the table that the dehomogenization approach performs
better than the traditional one.
\begin{table}[h]
\begin{tabular}{l|r|c|c||r|c|c}
   \hline
   &\multicolumn{3}{|c||}{With Dehomogenization}  &  \multicolumn{3}{c}{No Dehomogenization}   \\
   &\multicolumn{3}{|c||}{Relaxation (\ref{eq:Dmom:CP_nd})}    &  \multicolumn{3}{c}{Relaxation (\ref{eq:CPnondehom})}   \\ \hline
   Example & time  & accuracy & order & time & accuracy & order\\ \hline
   \ref{ex:CPmat}(A) & 1.03 & $1.38\cdot 10^{-6}$ & 3 &  2.77 & $1.71\cdot 10^{-6}$ & 3\\
   \ref{ex:CPmat}(B) & 0.52 & $1.97\cdot 10^{-6}$ & 2&  0.39 & $2.05\cdot 10^{-6}$ & 2\\
   \ref{ex:CPmat}(C) & 0.11 & not CP & 2 &  0.14 & not CP & 2\\
   \ref{ex:CPmemtensor} (i) & 0.22 & not CP & 4&  0.58 & not CP & 4\\
   \ref{ex:CPmemtensor}(ii) & 0.49 & $4.13\cdot 10^{-6}$ & 3& 1.05 & $1.54\cdot 10^{-6}$ & 3\\
   \ref{ex:CPtensor} (i) & 2.22 & $4.96\cdot 10^{-6}$ & 3 & 2.44 & $5.51\cdot 10^{-6}$ & 3\\
   \ref{ex:CPtensor}(ii) & 1.48 & $9.17\cdot 10^{-8}$ & 3 & 2.14 & $2.10\cdot 10^{-5}$& 4 \\
   \ref{ex:large} & 3.77 & $1.06\cdot 10^{-9}$ & 6 &  36.69 & $5.54\cdot 10^{-6}$ & 6\\ \hline
\end{tabular}
\caption{Comparison between relaxations (\ref{eq:Dmom:CP_nd}) and (\ref{eq:CPnondehom})}\label{tab:comparison}
\end{table}
}

\subsection{CP tensor approximations}
\label{ssc:CPapp}

For a given tensor $\mc{C}$, its best CP tensor approximation
is an optimizer of
\be
\label{opt:CPappro}
\left\{
\begin{array}{cl}
\min & \|\mc{X}-\mc{C}\|\\
\st & \mc{X}\in \mc{CP}_n^{\otimes d},
\end{array}
\right.
\ee
where $\|\cdot \|$ is the Hilbert-Schmidt norm given as in \reff{HSnm:||A||}.
The optimization \reff{opt:CPappro} is called the CP tensor approximation problem.
For the special case that $d=2$, it is reduced to be a CP-matrix approximation problem
(see the work \cite{FanZhou16CPmat,SpnDur14CPmat}).

The CP tensor approximation problem is equivalent to
a linear conic optimization problem with the CP tensor cone.
For each $\alpha\in \overline{\N}_d^n$, denote by $\mc{E}_{\alpha}\in \ttS^d(\re^n)$
the basis symmetric tensor whose corresponding tms
has zero entries except the $\alpha$th entry being $1$.
Then, the optimization \reff{opt:CPappro} is equivalent to
\begin{equation}
\label{opt:CPappro:lccp}
\left\{
\begin{array}{cl}
\min & w_0\\
\st & \sum\limits_{ \alpha \in \overline{\N}_d^n } w_{\alpha}
   \mc{E}_{\alpha}\in \mc{CP}_n^{\otimes d},\\
& \big\|\mc{C} - \sum\limits_{\alpha\in \bar{\N}_d^n}
   w_{\alpha} \mc{E}_{\alpha}  \big\|\le w_0,\\
& w = (w_{\alpha})\in \re^{ \overline{\N}_d^n },\, w_0 \in \re,
\end{array}
\right.
\end{equation}
which is in form of \reff{org:min<cy>:CPnd}.
The optimization \reff{opt:CPappro:lccp}
can be solved by Algorithm \ref{alg:linconc:CP}. Suppose it has the optimizer $(w_0^*, w^*, z^*)$.
Then the best CP-tensor approximation of $\mc{C}$ is
$
\sum_{\alpha \in \overline{\N}_d^n } w_{\alpha}^*\mc{E}_{\alpha}.
$
The following is an example to show the efficiency of Algorithm~\ref{alg:linconc:CP}
for solving CP tensor approximations.

\begin{example}
Consider the symmetric matrix
\[ C = \left[\begin{array}{rrrrr}
    1.0 & 2.0 &   1.5 & 0.0  &  2.5\\
    2.0 & 0.0 &  -1.0 & 2.0 & -2.5\\
    1.5 & -1.0 & -4.0 & 3.0 &  4.5\\
   0.0  &  2.0 & 3.0 & -2.0 & 1.0 \\
    2.5 & -2.5 & 4.5 & 1.0 & 0.0
    \end{array}\right]. \]
We use Algorithm \ref{alg:linconc:CP} to compute its best CP matrix approximation.
The algorithm terminates at $k=2$ with $0.22$ second. We solve for $w_0^*=9.6532$ and
\[\begin{array}{cl}
w^* = & (1.9059,\,0.9854,\,1.2192,\,0.9893,\,1.6969,\,1.2901,\,0.0000,\\
& \quad 0.4209,\,0.0000,\,1.2889,\,0.7060,\,1.7939,\,0.5240,\,0.9826,\,2.4969).
\end{array}
\]
The output best CP matrix approximation of $C$ is
\[ \bbm
1.9059  &  0.9854  &  1.2192 &   0.9893   & 1.6969\\
    0.9854  &  1.2901  &  0.0000  &  0.4209   & 0.0000\\
    1.2192  &  0.0000  &  1.2889  &  0.7060  & 1.7939\\
    0.9893  &  0.4209  &  0.7060  &  0.5240  &  0.9826\\
    1.6969 &   0.0000  &  1.7939  &  0.9826 &   2.4969
\ebm. \]
It can be decomposed as
\[
19.4584\bbm 0.2434\\0.0000\\0.2574\\0.1410
    \\0.3582\ebm^{\otimes 2}+
5.6356\bbm 0.3654\\0.4785\\0.0000\\0.1561
    \\0.0000\ebm^{\otimes 2}.
\]
\end{example}

\begin{example}
Consider the symmetric tensor $\mc{A}\in\ttS^3(\re^4)$ such that
\[ \mc{A}_{:,:,1} = \left[\begin{array}{rrrr}
  3  &   3  &   1  &  -3\\
  3  &   3  &  -1  &  -1\\
  1  &  -1  &   3  &   5\\
 -3  &  -1  &   5  &   3\end{array}\right],\quad
   \mc{A}_{:,:,2} = \left[\begin{array}{rrrr}
  3  &   3  &  -1  &  -1\\
  3  &   1  &   2  &   1\\
  -1  &   2   &  0  &   0\\
  -1  &   1   &  0  &   1\end{array}\right],\]
\[ \mc{A}_{:,:,3} = \left[\begin{array}{rrrr}
  1  &  -1  &   3  &   5\\
 -1  &   2  &   0  &   0\\
  3  &   0  &   2  &  -1\\
  5  &   0  &  -1  &   3\end{array}\right],\quad
   \mc{A}_{:,:,4} = \left[\begin{array}{rrrr}
  -3  &  -1  &   5  &   3\\
  -1  &   1  &   0  &   1\\
   5  &   0  &  -1  &   3\\
   3  &   1  &   3  &  -1\end{array}\right].\]
Here each $\mc{A}_{:,:,i}$ denotes the subtensor of $\mc{A}$ with labels
$(i_1,i_2,i)$ for all $1\le i_1,\, i_2\le 4$.
We use Algorithm \ref{alg:linconc:CP} to compute the best CP tensor approximation of $\mc{A}$.
The algorithm terminates at $k=3$ with $0.51$ second.
We solve for $w_0^*=14.2682$ and
\[\begin{array}{cl}
w^* =  & (4.1931,\,2.6035,\,1.5629,\,1.3894,\,2.3451,\,0.0293,\\
& \quad 0.0617,\,2.0098,\, 1.7390,\,1.6801,\,3.0183,\,0.6204,\,0.4050,\\
& \quad \quad 0.4299,\,0.1467,\,0.3084,\,2.9246,\,2.1886,\,2.1433,\,2.1677) .
\end{array}
\]
The output best CP tensor approximation of $\mc{A}$ is
\[
46.9502\bbm 0.2917\\0.0000\\0.3629
    \\0.3453 \ebm^{\otimes 3}+
6.2675\bbm 0.0669\\0.4393\\0.1592\\0.3346\ebm^{\otimes 3}+
3.1749\bbm 0.0000\\0.5429\\
    0.4571\\0.0000\ebm^{\otimes 3}
\]
\[+
1.3208\bbm 0.3566\\0.0000\\0.6434\\
    0.0000\ebm^{\otimes 3}+
19.5098\bbm 0.5337\\0.4663\\0.0000\\
    0.0000\ebm^{\otimes 3}.
\]
\end{example}

\subsection{CP tensor completions}

Let $\mc{I}$ be a nonempty label set for symmetric tensors
such that if $(i_1,\ldots, i_d)\in \mc{I}$
then all its permutations are also in $\mc{I}$.
Suppose $\mc{C}\in\ttS^d(\re^n)$ is a partially given symmetric tensor
such that  the entries $\mc{C}_{\alpha}\, (\alpha\not\in \mathcal{I})$ are given.
We look for the unknown entries $\mc{C}_{\alpha}\, (\alpha\in \mathcal{I})$
such that $\mc{C}$ is a CP tensor.
Let $\mc{C}_0$ be the symmetric tensor such that its entries are zeros
except $(\mc{C}_0)_{\alpha} = \mc{C}_{\alpha}$ for $\alpha \not\in \mathcal{I}$.
Consider the optimization problem
\be  \label{opt:CPcomple}
\left\{
\begin{array}{cl}
\min & \sum\limits_{\alpha\in \mc{I}}   w_{\alpha} \\
\st & \mc{C}_0 +  \sum\limits_{\alpha\in \mc{I}}
w_{\alpha}\mc{E}_{\alpha}\in \mc{CP}_{n}^{\otimes d},
\end{array}
\right. \ee
where $\mc{E}_{\alpha}$ denotes the basis symmetric tensor
whose corresponding tms has all zero entries except the $\alpha$th {entry being} {$1$}.
The \reff{opt:CPcomple} is a linear conic optimization problem with the CP tensor cone.
It aims to find a CP completion of $\mc{C}$ such that the sum of
{all} unknown entries is minimum. This question can be solved by
Algorithm~\ref{alg:linconc:CP}. If $w^* = (w_{\alpha}^*)_{\alpha\in \mc{I}}$
is a minimizer of \reff{opt:CPcomple}, then
$
\mc{C}_0 + \sum_{\alpha\in\mc{I}} w_{\alpha}^* \mc{E}_{\alpha}
$
is a CP tensor completion for $\mc{C}$.

\begin{example}\cite[Example 4.2]{ZhouFan14}
Consider the partially given matrix
\[ C = \bbm
* & 4 & 1 & 2 & 2\\
4 & * & 0 & 1 & 3\\
1 & 0 & * & 1 & 2\\
2 & 1 & 1 & * & 1\\
2 & 3 & 2 & 1 & *
\ebm, \]
where the symbol $*$ denotes the unknown entries. The label set
\[ \mc{I} = \{ (1,1),\, (2,2),\, (3,3),\, (4,4),\, (5,5) \}.
\]
We apply Algorithm \ref{alg:linconc:CP} to compute the CP matrix completion of $C$.
The algorithm terminates at the initial loop $k=1$ within $0.19$ second.
The optimal value and the optimal solution of \reff{opt:CPcomple} are
\[
F_{\mbox{min}} = 18.0039,\quad
w^* = (4.9100,\,6.0209,\,2.0774,\,1.0595,\,3.9360).
\]
The output CP matrix completion is
\[
\begin{bmatrix*}[l]
4.9100 & 4 & 1 & 2 & 2\\
4 & 6.0209 & 0 & 1 & 3\\
1 & 0 & 2.0774 & 1 & 2\\
2 & 1 & 1 & 1.0595 & 1\\
2 & 3 & 2 & 1 & 3.9360
\end{bmatrix*}
\]
which can be decomposed as follows
\[
9.9793\bbm 0.0000\\ 0.0000\\ 0.4046\\ 0.1008\\ 0.4946\ebm^{\otimes 2}+
9.3740\bbm 0.4902\\ 0.0000\\ 0.2176\\ 0.2907\\ 0.0015\ebm^{\otimes 2}+
32.6505\bbm 0.2853\\ 0.4294\\ 0.0000\\0.0713\\ 0.2140\ebm^{\otimes 2}.
\]
\end{example}

\begin{example}
Consider the partially given tensor $\mc{A}\in\ttS^3(\re^4)$
\[ A_{:,:,1} = \left[\begin{array}{cccc}
   *  &  3  &  6  &  5\\
   3  &  3  &  *  &  3\\
   6  &  *  &  6  &  4\\
   5  &  3  &  4  &  5\end{array}\right],\quad
   A_{:,:,2} = \left[\begin{array}{cccc}
   3  &  3  &  *  &  3\\
   3  &  *  &  5  &  5\\
   *  &  5  &  5  &  4\\
   3  &  5  &  4  &  5\end{array}\right],\]
\[ A_{:,:,3} = \left[\begin{array}{cccc}
   6  &  *  &  6  &  4\\
   *  &  5  &  5  &  4\\
   6  &  5  &  *  &  7\\
   4  &  4  &  7  &  7\end{array}\right],\quad
   A_{:,:,4} = \left[\begin{array}{cccc}
   5  &  3  &  4  &  5\\
   3  &  5  &  4  &  5\\
   4  &  4  &  7  &  7\\
   5  &  5  &  7  &  *\end{array}\right].\]
In the above, the symbol $*$ denotes the unknown entries, and each $\mc{A}_{:,:,i}$
denotes the subtensor of $\mc{A}$ with labels $(i_1,i_2,i)$ for all $1\le i_1,\, i_2\le 4$.
The label set of unknown entries is
\[ \mc{I} = \{(1,1,1),\ (1,2,3),\ (2,2,2),\ (3,3,3),\ (4,4,4)\}.
\]
We apply Algorithm \ref{alg:linconc:CP} to compute the CP tensor completion of $\mc{A}$.
The algorithm terminates at the loop $k=3$ within $0.20$ second.
The optimal value and the optimal solution of \reff{opt:CPcomple} are
\[
F_{\mbox{min}} = 40.7663,\quad
w^* = (7.9006,\,1.3225,\,6.8038,\,9.9248,\,8.2023).
\]
The output CP tensor completion can be expressed as
\[\begin{array}{c}
5.4917\bbm 0.0000  \\ 0.0000  \\  0.4085 \\ 0.5915\ebm^{\otimes 3}+
59.8411\bbm  0.0000\\ 0.3584\\ 0.3664\\0.2752\ebm^{\otimes 3}+
5.9400\bbm 0.0000\\0.5492\\0.4508\\0.0000\ebm^{\otimes 3}+\\
111.6693\bbm 0.2150\\0.2213\\0.2412\\0.3225\ebm^{\otimes 3}+
59.1666\bbm 0.3573\\0.0000\\0.3653\\0.2774\ebm^{\otimes 3}+
35.4546\bbm 0.3763\\0.3692\\0.0000\\0.2545\ebm^{\otimes 3}+\\
15.1997\bbm 0.5280\\0.0000\\0.4720\\0.0000\ebm^{\otimes 3}.
\end{array}
\]
\end{example}

\section{Conclusions and Discussions}

This paper proposes a dehomogenization approach for
studying completely positive tensors.
We give a hierarchy of Moment-SOS relaxations for approximating CP tensor cones,
based the dehomogenization. This helps us to get a Moment-SOS algorithm
(i.e., Algorithm~\ref{sdpalg:CPnd})
for checking memberships of CP tensor cones.
Moreover, we also give an algorithm for solving linear conic optimization
with CP tensor cones, using Moment-SOS relaxations based on dehomogenization.
The dehomogenization approach is more efficient than the traditional one.
Numerical experiments are given to show the efficiency of the proposed approach.

We would like to remark that the dehomogenization approach can also be used to check the memberships of copositive tensors and the related linear coinc optimization problems.
It is worthy to note that a linear conic optimization problem
with the copositive tensor cone is usually solved together with its dual,
which is a linear conic optimization problem with the CP tensor cone.
This can be seen as in \reff{org:min<cy>:CPnd}-\reff{org:max<ap>:COPnd}.
Therefore, the dehomogenization Algorithm~\ref{alg:linconc:CP} 
can also be used to solve linear conic optimization problems 
with the copositive tensor cone.

For a symmetric tensor $\mc{B}\in\ttS^d(\re^n)$, 
let $f = \mc{B}(x)$ be as in \reff{df:A(x)}.
The tensor $\mc{B}$ is copositive if and only if the form $f$ is copositive.
That is, $f(x)\ge 0$ for all $x\in\Dt$.
Recall the dehomogenization map $\varpi$ as in \reff{diamond:x:f}.
It holds that
\[
\mc{B}\in\mc{COP}_n^{\otimes d} \quad  \Leftrightarrow \quad 
f\in\mc{COP}_{n,d} \quad  \Leftrightarrow \quad 
\varpi(f)\in\mathscr{P}_d(\overline{\Dt}).
\]
Therefore, approximations for the positive polynomial cone 
$\mathscr{P}_d(\overline{\Dt})$ 
can be used to approximate copositive cones 
$\mc{COP}_d^{\otimes d}$ and $\mc{COP}_{n,d}$.
Consider the polynomial optimization
\begin{equation}
\label{eq:copmem}
\left\{\begin{array}{rl}
f_0 = \min & \varpi(f)(\bar{x})\\
\st & g(\bar{x})\ge 0.
\end{array}
\right.
\end{equation}
In the above,
the $g(\bar{x})\,\coloneqq\,(x_1,\,x_2,\dots,\, x_{\bar{n}},\, 1-e^T\bar{x})$.
It is clear that $\varpi(f)\in \mathscr{P}_d(\overline{\Dt})$ 
if and only if the optimal value $f_0\ge 0$.
By \cite{Nie19tight}, one can construct a tight hierarchy of Moment-SOS relaxations 
for solving \reff{eq:copmem}, using optimality conditions 
and Lagrange multiplier expressions.
Denote the polynomial tuples 
\[
\lmd = \bbm 
\nabla (\varpi(f))-\bar{x}^T\nabla (\varpi(f))\cdot e, \\
 -\bar{x}^T\nabla (\varpi(f)) 
 \ebm,\quad
h = \bbm \lmd_1 g_1 \\ \vdots \\ \lmd_n g_n \ebm,
\]
where $\nabla$ denotes the gradient with respect to $\bar{x}$,
and $g_i$ (resp., $\lmd_i$) denotes the $i$th entry 
of the polynomial tuple $g$ (resp., $\lmd$).
The $k$th order tight relaxation for \reff{eq:copmem} is
\begin{equation}
\label{eq:dhcop}
\left\{
\begin{array}{rl}
f_k = \max & \gamma\\
\st & \varpi(f)-\gamma\in \qmod{g,\lmd}_{2k}
+ \ideal{h}_{2k}.
\end{array}
\right.
\end{equation}
A sufficient condition for $\mc{B}$ to be copositive is that $f_k\ge 0$ 
for some relaxation order $k$.
We refer to \cite{Nie19tight,Nie18cop} for more details 
about the construction of tight relaxations.
For instance, consider the Horn matrix (see \cite{Horn})
\[
H = \left[\begin{array}{rrrrr}
1 & -1 & 1 & 1 & -1\\
-1 & 1 & -1 & 1 & 1\\
1 & -1 & 1 & -1 & 1\\
1 & 1 & -1 & 1 & -1\\
-1 & 1 & 1 & -1 & 1
\end{array}\right],
\]
which is known to be copositive.
It corresponds to the form
\[
f(x) = x^Tx+2x_1(x_3+x_4-x_2-x_5)+2x_2(x_4+x_5-x_3)+2x_3(x_5-x_4)-2x_4x_5.
\]
After dehomogenization, we get
\[
\varpi(f) = 1-4x_1-4x_4+4x_1^2+4x_1x_3+8x_1x_4-4x_2x_3+4x_2x_4+4x_4^2.
\]
For the hierarchy of relaxations~\reff{eq:dhcop}, 
when $k=2$, we get $f_2 \approx -1.4681\cdot 10^{-7}$.

\medskip \noindent
{\bf Acknowledgement}
The authors are partially supported by NSF grant DMS-2110780.

\end{document}